\documentclass[10pt,oneside]{article}

\usepackage[utf8]{inputenc} %
\usepackage[T1]{fontenc}    %
\usepackage{url}            %
\usepackage{booktabs}       %
\usepackage{amsfonts}       %
\usepackage{nicefrac}       %
\usepackage{microtype}      %

\usepackage{lmodern}
\usepackage{times}

\usepackage{amssymb,amsmath,amsthm}
\usepackage{bbold}
\usepackage[short]{optidef}

\usepackage{framed,multirow,multicol}

\usepackage{xcolor,xspace}
\usepackage{lscape}
\usepackage{subfigure,graphicx,epsfig,tikz,caption}
\usepackage[normalem]{ulem}
\usepackage{enumerate}
\usepackage{verbatim}
\usepackage{makeidx,latexsym}
\usepackage[colorlinks=true,citecolor=blue]{hyperref}%

\usepackage{bm}
\usepackage{stmaryrd}
\usepackage{lscape,rotating}

\usepackage[capitalize]{cleveref}
\crefrangeformat{equation}{(#3#1#4)--(#5#2#6)}

\usepackage[numbers,sort&compress,square,comma]{natbib}

\usepackage{algorithmic,algorithm}

\usepackage[english]{babel}
\usepackage{bbm}
\usepackage{pgfplots}

    \usepackage{parskip}
\usepackage[letterpaper,margin=1.1in]{geometry}

    \usepackage{thmtools,thm-restate}

    \declaretheorem{lemma}

    \declaretheoremstyle[qed=$\square$]{definitionwithend}

    \declaretheorem[style=definitionwithend]{remark}
\PassOptionsToPackage{numbers, compress}{natbib}

\DeclareMathOperator{\DR}{DR}
\DeclareMathOperator{\SAA}{SAA}

\DeclareMathOperator{\dist}{dist}
\DeclareMathOperator{\intt}{int}
\DeclareMathOperator{\cl}{cl}

\DeclareMathOperator{\Proj}{Proj}

\newcommand{\vc}{\mathbf{c}}

\newcommand{\va}{\mathbf{a}}
\newcommand{\vA}{\mathbf{A}}

\newcommand{\vr}{\mathbf{r}}

\newcommand{\vz}{\mathbf{z}} %
\newcommand{\vx}{\mathbf{x}} %
\newcommand{\vy}{\mathbf{y}} %
\newcommand{\vpi}{\bm{\pi}} %
\newcommand{\vlambda}{\bm{\lambda}} %

\newcommand{\vmu}{\bm{\mu}} %
\newcommand{\vb}{\mathbf{b}} %

\newcommand{\vu}{\mathbf{u}}

\newcommand{\vrho}{\bm{\rho}}

\newcommand{\vxi}{\bm{\xi}}

\newcommand{\bbE}{\mathbb{E}}

\newcommand{\bbP}{\mathbb{P}}

\newcommand{\bbR}{\mathbb{R}}

\newcommand{\R}{\mathbb{R}}

\newcommand{\cF}{\mathcal{F}}

\newcommand{\cS}{\mathcal{S}}

\newcommand{\cX}{\mathcal{X}}

\newcommand{\epr}{\hfill\hbox{\hskip 4pt \vrule width 5pt height 6pt depth 1.5pt}\vspace{0.0cm}\par}

\def\true{1}

\def\flagJournal{0} %

\newcommand {\beqn}{\begin{equation}}\newcommand {\eeqn}{\end{equation}}
\newcommand {\beqan}{\begin{eqnarray}}\newcommand {\eeqan}{\end{eqnarray}}
\newcommand {\beqa}{\begin{eqnarray*}}\newcommand {\eeqa}{\end{eqnarray*}}
\newcommand{\skipit}  [1] {}

\newtheorem{theorem}{Theorem}%

\begin{document}
\title{Strong Formulations for Distributionally Robust Chance-Constrained Programs with Left-Hand Side Uncertainty under Wasserstein Ambiguity}
\author{
	Nam Ho-Nguyen\thanks{Discipline of Business Analytics, The University of Sydney Business School, Sydney, Australia, \url{nam.ho-nguyen@sydney.edu.au}}
	\and
	Fatma K{\i}l{\i}n\c{c}-Karzan\thanks{Tepper School of Business, Carnegie Mellon University, Pittsburgh, PA 15213, USA, \url{fkilinc@andrew.cmu.edu}}
	\and 
	Simge K\"{u}\c{c}\"{u}kyavuz\thanks{Industrial Engineering and Management Sciences, Northwestern University, Evanston, IL 60208, USA, \url{simge@northwestern.edu}}
	\and
	Dabeen Lee\thanks{Discrete Mathematics Group, Institute for Basic Science (IBS), Daejeon 34126, Republic of Korea, \url{dabeenl@ibs.re.kr}}
}	

\date{\today}

\maketitle

\begin{abstract}
Distributionally robust chance-constrained programs (DR-CCP) over Wasserstein ambiguity sets exhibit attractive out-of-sample performance and admit big-$M$-based mixed-integer programming (MIP) reformulations with conic constraints. However, the resulting formulations  often suffer from scalability issues as sample size increases. To address this shortcoming, we derive stronger formulations that scale well with respect to the sample size. Our focus is on ambiguity sets  under the so-called left-hand side (LHS) uncertainty, where the uncertain parameters affect the coefficients of the decision variables in the linear inequalities defining the safety sets.  The interaction between the uncertain parameters and the variable coefficients in the safety set definition causes  challenges in strengthening the original big-$M$ formulations.  By exploiting the connection between nominal chance-constrained programs and DR-CCP, we obtain strong formulations with significant enhancements. In particular, through this connection, we derive a linear number of valid inequalities, which can be immediately added to the formulations to obtain improved  formulations in the original space of variables. In addition, we suggest a quantile-based strengthening procedure that allows us to reduce the big-$M$ coefficients drastically.  Furthermore, based on this procedure,   we propose an exponential class of inequalities that can be separated efficiently within a branch-and-cut framework. The quantile-based strengthening procedure can be expensive. Therefore, for the special case of covering and packing type problems, we identify an efficient scheme to carry out this procedure. We demonstrate the computational efficacy of our proposed formulations  on two classes of problems, namely  stochastic portfolio optimization and resource planning.

\end{abstract}

\section{Introduction}\label{sec:intro}

Chance-constrained programming is an important paradigm in optimization under uncertainty. It acknowledges that it may not be possible to satisfy all constraints of a system due to the inherent uncertainty in the model parameters; instead, it aims to satisfy the system constraints with high probability.  The generic  form of a chance-constrained  program (CCP)  is given by 
\begin{equation}\label{eq:ccp}
\min_{\vx} \left\{\vc^\top \vx:~  \vx \in \mathcal{X}, ~ \bbP^*[\vxi \not\in \cS(\vx)] \leq \epsilon \right\}. \tag{CCP}
\end{equation}
Here, $\cX \subset \bbR^L$ is a domain for the vector of decision variables $\vx$,  $\vxi \in \bbR^K$ is a random vector distributed according to $\bbP^*$, and $\epsilon \in (0,1)$ is the risk tolerance for the random variable $\vxi$ falling outside a decision-dependent safety set given by $\cS(\vx) \subseteq \bbR^K$. 

A main challenge in formulating and solving~\eqref{eq:ccp}~problems is that the  distribution $\bbP^*$ is typically unknown or else is not efficiently computable in high dimensions. Often, in practice, this issue is addressed by approximating $\bbP^*$ with an \emph{empirical distribution} $\bbP_N$ obtained by  sampling $N$ independent and identically distributed (i.i.d.) samples $\{ \vxi_i \}_{i \in [N]}$  from $\bbP^*$, where $[N]:=\{1,\ldots,N\}$. This leads to the natural and popular \emph{Sample Average Approximation} (SAA) formulation obtained by replacing $\bbP^*$  with  $\bbP_N$ in~\eqref{eq:ccp} (see \cref{sec:nominal}). 
This procedure has been shown to be statistically consistent \citep{Cal1,campi,sample}, but is also known to be quite sensitive to the samples drawn unless $N$ is quite large, in which case the resulting formulation is computationally intractable.  
Consequently, finding a solution to~\eqref{eq:ccp} that is robust to errors in approximating $\bbP^*$ with $\bbP_N$ is of interest. To address this issue, a recent growing stream of research studies the following \emph{distributionally robust chance-constrained program}~\eqref{eq:dr-ccp}:
\begin{equation}\label{eq:dr-ccp}
\min_{\vx} \left\{\vc^\top \vx:~  \vx \in \mathcal{X}, ~ \sup_{\bbP \in \cF} \bbP[\vxi \not\in \cS(\vx)]\leq \epsilon \right\}, \tag{DR-CCP}
\end{equation}

where $\cF$ is an ambiguity set of distributions on $\bbR^K$. 
In~\eqref{eq:dr-ccp}, the set $\cF$ plays a critical role. Usually, $\cF:=\cF_N(\theta)$ with a parameter $\theta>0$ is selected such that the empirical distribution $\bbP_N$ is contained in it, and $\theta$ governs the size of the ambiguity set (and consequently the degree of conservatism of~\eqref{eq:dr-ccp}). See \citet{Rahimian2019DistributionallyRO} and references therein for a survey on distributionally robust optimization, in particular, the properties of~\eqref{eq:dr-ccp} and existing solution methods. 

One of the most commonly studied set $\cF_N(\theta)$ is the so-called \emph{Wasserstein ambiguity set}, which 
is defined by the Wasserstein distance ball of radius $\theta$ around the empirical distribution $\bbP_N$. The Wasserstein ambiguity set gained popularity due to its desirable statistical properties and advantages over  other ambiguity sets based on moments, $\phi$-divergences, unimodality, or support; see e.g., \cite{EOO03,JG16,calafiore:jota06,Hanasusanto2015,ZhangJiangShen2018,Li2019,Xie18}. Furthermore, Wasserstein uncertainty set is also attractive because  the dual representation for the worst-case probability $\bbP[\vxi \not\in \cS(\vx)]$ over the  ambiguity set $\bbP \in \cF_N(\theta)$ \cite{gao2016distributionally,BlanchetMurthy2019,MohajerinEsfahaniKuhn2018} can be used to derive deterministic non-convex reformulations of \eqref{eq:dr-ccp}~\cite{HotaEtAl2019,xie2018distributionally,chen2018data}. For example, for certain linear forms of safety sets $\cS(\cdot)$, \citet{chen2018data} and \citet{xie2018distributionally} show that  \eqref{eq:dr-ccp} can be represented as a mixed-integer program (MIP) with big-$M$ coefficients, which enables, in theory, modeling and solving these problems  with black-box solvers.  In practice, however, the resulting MIPs even with moderate sample sizes (e.g., $N=100$) cannot be solved in reasonable time  with commercial MIP solvers. 
 
In the literature, the scalability challenge of~\eqref{eq:dr-ccp} with Wasserstein ambiguity is addressed by exploiting further problem structures. %
\citet{xie2018distributionally} considers the case where all the decision variables in~\eqref{eq:ccp} are binary, for which he derives a big-$M$-free formulation that leads to notable computational benefits. \citet{WangLiMehrotra2020} impose the assumption that the support of $\vxi$ is finite and all the decision variables are binary when formulating distributionally robust assignment problems with the so-called left-hand side (LHS) uncertainty, in which the uncertain parameters $\vxi$ affect the coefficients of the decision variables in the safety set $\cS(\vx)$. They further assume that chance constraints are given \emph{individually}, i.e., each chance constraint takes a single inequality.
\citet{ZhangDong2020} also use individual chance constraints to model
an uncertain renewable load control problem with binary variables and the right-hand side (RHS) uncertainty structure, where the uncertain parameters do not interact with the coefficients of the variables in the safety set, and propose some enhancements to the MIP reformulation. \citet{JiLejeune2019} give MIP formulations of \eqref{eq:dr-ccp}  under %
other structural assumptions on the support of $\vxi$. In contrast to these work, we do not assume or exploit binary problem structure,  place no assumptions on the support of $\vxi$, and we consider \emph{joint} chance constraints under LHS uncertainty, i.e., a chance constraint may take a system of inequalities.

In our previous work~\cite{rhs2020}, we observe that the SAA formulation can be cast as   
 \cref{eq:dr-ccp} with radius $\theta=0$, since $\cF_N(0)=\{\bbP_N\}$ under Wasserstein ambiguity, and that for $\theta>0$, the SAA formulation is a relaxation of \cref{eq:dr-ccp}. We then
 exploit this connection between \cref{eq:dr-ccp} and SAA to address the (easier) case of  RHS uncertainty. Our proposed approach in~\cite{rhs2020} provides stronger  formulations and valid inequalities for \cref{eq:dr-ccp}, which are instrumental in solving both instances that are an order of magnitude larger than those in the literature,  from 100s of scenarios to  1000s of scenarios, and instances that are difficult even for small number of scenarios (100s) due to the number of original decision variables and their problem structure. In this paper, we further explore this connection between \cref{eq:dr-ccp} and SAA to address the more difficult case of  LHS  uncertainty. As a result of the more complex structure of LHS uncertainty, our developments here differ from the RHS uncertainty case in  \cite{rhs2020}, and we highlight these wherever relevant.
 
\subsection*{Contributions and outline}
In~\cref{sec:problem} we formally describe our problem and the safety sets of interest.  Our main contributions are summarized as follows.
\begin{itemize}
\item  We delineate the relationship between SAA and~\eqref{eq:dr-ccp} under Wasserstein ambiguity with LHS uncertainty in \cref{sec:nominal}. We note that while recognizing that SAA is a relaxation of DR-CCP is not novel in itself, the main innovation of our work is to build a \emph{precise link} between the CCP formulation and the DR-CCP formulation in terms of the binary variables used in the CCP formulation. Using this connection, we obtain %
a stronger formulation for~\eqref{eq:dr-ccp} by adding linearly many valid inequalities to the standard MIP formulation from the literature~\cite{chen2018data,xie2018distributionally}, given in~\eqref{eq:joint}, \emph{without} increasing the number of variables. %
Our formulation~\cref{eq:joint-knapsack} is an exact reformulation of~\eqref{eq:dr-ccp} %
for the closed safety sets %
and gives a tighter relaxation %
for the open safety sets; %
see \cref{thm:knapsack-valid-joint}.
We then exploit this link to further strengthen the DR-CCP formulation (\cref{sec:quantile}).

Our previous paper \citep{rhs2020} also provides a link between CCP and DR-CCP formulations in the \emph{right-hand side uncertainty} case, yet the left-hand side uncertainty case that we consider in the present paper is considerably more complicated. The main difference can be seen through  \cref{lemma:joint-valid} that states that the formulation of \eqref{eq:dr-ccp} in equation~\eqref{eq:joint} from previous literature is not exact in the case of LHS uncertainty. %
However, in the case of RHS uncertainty the associated MIP formulations are exact; see \cref{rem:joint-valid-RHS-compare}. \cref{sec:nominal} aims to  clarify this distinction. %

\citet[Theorem 3]{xie2018distributionally} provides a better relaxation of~\eqref{eq:dr-ccp} than the ordinary CCP by using the CVaR interpretation of the DR-CCP formulation and its value-at-risk relaxation. The author does not use this relaxation to derive an exact reformulation of DR-CCP---doing so would require a new set of binary variables and associated big-M constraints. In contrast,  it bears repeating that our aim is \emph{not} only to provide a stronger relaxation, but also to link the CCP and DR-CCP formulations without introducing new variables, and exploit this link to solve the \emph{exact} DR-CCP formulation effectively. 
	In \cref{sec:other_relaxations}, we present an argument that the relaxation in \citep[Theorem 3]{xie2018distributionally} cannot be linked with the DR-CCP formulation in the same manner as what we do in this paper.  Similarly, \citet{ChenXie2020} also use the observation that SAA is a relaxation of DR-CCP, but they do not reveal the connection regarding the binary variables.

\item In \cref{sec:quantile}, we exploit this relationship between SAA and~\eqref{eq:dr-ccp} to suggest a further quantile-based strengthening of the formulation. %
 In particular, the connection with SAA exposes a \emph{mixing substructure} in the MIP reformulation of \eqref{eq:dr-ccp}. For the RHS uncertainty case, exploiting this mixing substructure to reduce the big-$M$ coefficients entails simply the sorting of the nominal right-hand side parameters. Due to the interaction of the random parameters with the decision variables, this procedure cannot be immediately applied to the case of LHS uncertainty (see \cref{rem:mixing-RHS-comparison} for a discussion of the differences in this procedure against our previous paper~\citep{rhs2020}).
 Hence, for \eqref{eq:dr-ccp} with LHS uncertainty we suggest a more involved quantile-based strengthening framework in~\cref{sec:quantile}, which lets us derive a further improved formulation~\eqref{eq:joint-k-reduced}.
As opposed to the previous literature on quantile-based strengthening, our results exploit a unique structure stemming from the MIP formulation of \eqref{eq:dr-ccp}.  
Through these developments, we are able to perform significant coefficient strengthening  and enhance our MIP formulations with an additional exponential class of valid inequalities;  see \cref{thm:improved-formulation} and inequalities~\eqref{eq:mixing}. 
We note that the mixing procedure has  been applied in the context of basic CCP before; see e.g.,  \citep{luedtke2010integer,luedtke2014branch-and-cut}. In this paper, we do utilize it again, but in the new context of strengthening~\eqref{eq:dr-ccp}.

\item In certain cases, the most powerful version of quantile strengthening procedure to generate the coefficients of the mixing set may require us to solve a number of subproblems, which are sometimes large optimization problems themselves. In the most general case, we would need $N^2$ calls to an LP solver to compute these coefficients. In \cref{sec:cover-pack}, we consider the special case of covering and packing constraints and show that these special structures enable us to develop a more efficient coefficient strengthening procedure which does not require us to use an LP solver. This expedites the process of deriving formulation~\eqref{eq:joint-k-reduced} for the case of covering and packing constraints. We note that these classes of problems are sufficiently prevalent \citep{bicriteria,binary-packing,ZhangJiangShen2018} in this literature as well.

\item Finally, in  \cref{sec:experiments},  we assess the computational impact of our theoretical developments on stochastic portfolio optimization and resource planning problems. We have conducted extensive numerical experiments to demonstrate the effectiveness of our proposed approach. Our numerical results show that our formulation significantly improves upon the existing formulations.

For the portfolio optimization problem, we test instances with $N\in\{100,300,500,1000\}$ samples. We observe that our framework reduces the overall solution time remarkably compared to the existing formulations, regardless of sample size. In particular, when $N\in\{500,1000\}$, we see that while none of the instances can be solved to optimality within one hour with the existing formulations, our proposed approach attain an optimal solution for most of the instances within a couple of minutes. The instances with $N\in\{100,300\}$ are easier, but we still observe that our formulation performs much better.

The resource planning problem instances are much harder than the portfolio optimization instances, but we still obtain similar results showing the efficacy of our proposed approach against the existing formulations. For the resource planning problem, even with  sample sizes as small as $N=100$, 75 out of the 100 instances are not solved to optimality, and the basic formulation terminates  with 45-77\% optimality gap after an hour.  Furthermore, we show that none of the instances with $N=300$ can be solved to optimality with the basic formulation and the algorithm  terminates with over 90\% optimality gap after an hour of computing with the existing formulations. For this harder problem class our formulation solves 98 (instead of 25) of  the 200 total  instances are solved to optimality and the largest optimality gap is 14\%.

 \end{itemize}

\section{Problem formulation}\label{sec:problem}

We consider Wasserstein ambiguity sets $\cF_N(\theta)$ defined as the $\theta$-radius Wasserstein ball of distributions on $\bbR^K$ around  the empirical distribution $\bbP_N$. We use the \emph{1-Wasserstein distance}, based on a norm $\|\cdot\|$, between two distributions $\bbP$ and $\bbP'$. This is defined as follows:
\begin{equation}\label{eq:wasserstein-dist}
d_W(\bbP,\bbP') := \inf_{\Pi} \left\{ \bbE_{(\vxi,\vxi') \sim \Pi}[\|\vxi - \vxi'\|] : \Pi \text{ has marginal distributions } \bbP, \bbP' \right\}.
\end{equation}
Then, the Wasserstein ambiguity set is 
\[\cF_N(\theta) := \left\{ \bbP : d_W(\bbP_N,\bbP) \leq \theta\right\}. \] 
Given a decision $\vx \in \cX$ and random realization $\vxi \in \bbR^K$, the \emph{distance from $\vxi$ to the unsafe set, $\R^K\setminus \cS(\vx)$,} is defined as
\begin{equation}\label{eq:distance}
\dist(\vxi,\cS(\vx)) := \inf_{\vxi' \in \bbR^K} \left\{ \|\vxi - \vxi'\| : \vxi' \not\in \cS(\vx) \right\}.
\end{equation}
It is important to highlight that here $\dist(\vxi,\cS(\vx))$ computes the distance from $\vxi$ to the unsafe set, \emph{not to the set $\cS(\vx)$}. Despite this, we chose to use this notation to emphasize that the distance function depends on the realization of the random parameter $\vxi$ and the decision-dependent safety set $\cS(\vx)$.
In the remainder of the paper, %
we assume that the sample $\{\vxi^i\}_{i \in [N]}$, the risk tolerance $\epsilon \in (0,1)$ and the radius $\theta > 0$ are fixed. 
We denote the feasible region of \eqref{eq:dr-ccp} as follows:
\begin{equation}
\cX_{\DR}(\cS) := \left\{ \vx \in \cX :~ \sup_{\bbP \in \cF_N(\theta)} \bbP[\vxi \not\in \cS(\vx)] \leq \epsilon \right\}.\label{eq:dr-ccp-region}
\end{equation}
In this  notation we make the dependence on the safety set  function $\cS$ explicit because the valid inequalities we will introduce have an explicit dependence on the safety set. 

It was recently shown that the distributionally robust chance constraint in \eqref{eq:dr-ccp}, and therefore~\eqref{eq:dr-ccp-region}, can be reformulated in a computationally tractable form~\cite{chen2018data,xie2018distributionally}. These reformulation results are obtained based on earlier developments on duality theory for Wasserstein distributional robustness~\cite{BlanchetMurthy2019,gao2016distributionally}. More precisely, whenever $\cS(\vx) \subseteq \bbR^K$ is an arbitrary open set for each $\vx \in \cX$ and $\theta > 0$, \citet[Theorem 3]{chen2018data} show that $\cX_{\DR}(\cS)$ can be formulated as \begin{align}\label{eq:cc-distance-formulation}
	\cX_{\DR}(\cS) = \left\{ \vx \in \cX : \begin{aligned}
		&\quad \exists\  t \geq 0, \ \vr \geq \bm{0},\\
		&\quad \dist(\vxi^i,\cS(\vx)) \geq t - r^i, \ i \in [N],\\
		&\quad \epsilon\, t \geq \theta + \frac{1}{N} \sum_{i \in [N]} r^i
	\end{aligned} \right\} .
\end{align}
Under the additional restriction $t > 0$, another similar formulation also holds. \citet[Proposition 1]{xie2018distributionally} derives the same formulation for the case of \emph{closed linear safety sets}, which we will define later in this section. In fact, \eqref{eq:cc-distance-formulation} holds regardless of whether $\cS(\vx)$ is open or closed because  \citet[Proposition 3]{gao2016distributionally} show that for given $\cS(\vx)$,
\[
\sup_{\bbP \in \cF(\theta)} \bbP[\vxi \not\in \intt \cS(\vx)] = \sup_{\bbP \in \cF(\theta)} \bbP[\vxi \not\in \cS(\vx)] = \sup_{\bbP \in \cF(\theta)} \bbP[\vxi \not\in \cl \cS(\vx)]
\] and 
\begin{equation}\label{eq:dist-invariant}
\dist(\vxi, \intt \cS(\vx)) = \dist(\vxi, \cS(\vx)) = \dist(\vxi, \cl \cS(\vx)),
\end{equation} 
where $\intt \cS(\vx)$ and $\cl \cS(\vx)$ denote the interior and closure of $\cS(\vx)$, respectively. To implement formulation~\eqref{eq:cc-distance-formulation} for~\eqref{eq:dr-ccp}, it is crucial to represent the constraints $\dist(\vxi^i,\cS(\vx)) \geq t - r^i$  in a computationally tractable form. To do so, it is important to understand the distance function $\dist(\vxi,\cS(\vx))$, which depends on not only the random parameter $\vxi$ and the decision vector $\vx$ but also the structure of the safety set.

Of particular interest is \emph{linear safety sets}, defined by a finite set of linear inequalities.  
In this case,  the decision-dependent safety set $\cS(\vx)$ is of the form of either $\cS^o(\vx)$ for \emph{open} safety sets  or $\cS^c(\vx)$ for  \emph{closed} safety sets, respectively, where
\begin{subequations}\label{eq:safety}
\begin{align}
\cS^o(\vx) &:= \left\{ \vxi :~ (\vb-\vA^\top \vx)^\top \vxi_p + d_p - \va_p^\top \vx > 0, \ p \in [P] \right\},\label{eq:safety-open}\\
\cS^c(\vx) &:= \left\{ \vxi :~ (\vb-\vA^\top \vx)^\top \vxi_p + d_p - \va_p^\top \vx \geq 0, \ p \in [P] \right\}.\label{eq:safety-closed}
\end{align}
\end{subequations}
As both open and closed safety sets have attracted attention in the recent literature; namely, \citet{chen2018data} consider open safety sets, while \citet{xie2018distributionally} consider closed safety sets. We will examine both types of safety sets in this paper as well. As we will see in \cref{sec:nominal}, the addition of the CCP constraints to the~\eqref{eq:dr-ccp} formulation has a different impact on the exactness of the resulting formulation depending on whether the set is open or closed. In~\eqref{eq:safety}, $P$ determines the number of inequalities defining the linear safety set. When $P=1$, we refer to the chance constraint $\bbP^*[\vxi \not\in \cS(\vx)] \leq \epsilon$ as an \emph{individual chance constraint}, and when $P>1$, we call $\bbP^*[\vxi \not\in \cS(\vx)] \leq \epsilon$ as a \emph{joint chance constraint}. The random vector $\vxi$ in~\eqref{eq:safety} consists of $P$ subvectors $\vxi_1,\ldots,\vxi_P$, each of which is associated with an inequality in the safety set description. The classical literature on CCP typically considers different settings depending on whether or not the linear inequalities have uncertainty in the coefficients of the decision variables. When $\vA\neq\bm{0}$, we say that the chance constraint has \emph{left-hand side (LHS) uncertainty}. When $\vA=\bm{0}$, then we have $\vb\neq\bm{0}$ so that inequalities have random data, in which case, we say that the chance constraint has \emph{right-hand side (RHS) uncertainty}. 
When $A,b \neq 0$, this model considers the most general case of both LHS and RHS uncertainty simultaneously. However, following the standard terminology, we will refer to this general model as the LHS uncertainty case.

In this paper, we focus on linear safety sets given by~\eqref{eq:safety} with LHS uncertainty. In the case of linear safety sets of form~\eqref{eq:safety},  formulation~\eqref{eq:cc-distance-formulation} admits a tractable reformation of~\eqref{eq:dr-ccp}. \citet{chen2018data} focus on the open safety set $\cS^o(\vx)$ given by~\eqref{eq:safety-open}, for which they provide a MIP reformulation. Independently, \citet{xie2018distributionally} considers the closed safety set $\cS^c(\vx)$ given by~\eqref{eq:safety-closed} and arrive at a MIP reformulation that is almost identical to the MIP formulation of \citet{chen2018data}. When $\cS(\vx)=\cS^o(\vx)$ and $\vb\neq \vA^\top \vx$, \citet[Lemma 2]{chen2018data} prove that the associated distance function is given by
\begin{equation}
\dist(\vxi,\cS(\vx)) = \max\left\{ 0,\ \min_{p \in [P]} \frac{(\vb-\vA^\top \vx)^\top \vxi_p + d_p - \va_p^\top \vx}{\|\vb-\vA^\top \vx \|_*} \right\},\label{eq:distance-linear}
\end{equation}
where $\|\cdot\|_*$ represents the norm dual to $\|\cdot\|$. 
\citet[Theorem 1]{xie2018distributionally} argues that~\eqref{eq:distance-linear} holds even when $\cS(\vx)=\cS^c(\vx)$ and $\vb\neq \vA^\top \vx$; note that this can also be deduced from~\eqref{eq:dist-invariant}. On the other hand, if $\vb = \vA^\top \vx$, we must compute the distance function through the original definition given by~\eqref{eq:distance}, and its characterization differs depending on whether we consider $\cS^o(\vx)$ or $\cS^c(\vx)$. We note that this issue is not present in the case of RHS uncertainty because $\vb\neq\bm{0}$ and $\vA=\bm{0}$ automatically imply that $\vb \neq \vA^\top \vx$ for all $\vx$, and so the distance function is precisely~\eqref{eq:distance-linear} for all $\vx$.

Assuming that $\vb \neq \vA^\top \vx$ for any $\vx \in \cX$, we can substitute the formula~\eqref{eq:distance-linear} for the distance function in the reformulation~\eqref{eq:cc-distance-formulation} of~\eqref{eq:dr-ccp} with joint chance constraints. Due to the max-terms in~\eqref{eq:distance-linear}, the resulting formulation is non-convex. Nevertheless,  under the assumption that $\vb \neq \vA^\top \vx$ for any $\vx \in \cX$, by introducing binary variables and big-$M$ constraints to model the distances and the max-terms therein,  \citet[Theorem 2]{xie2018distributionally} obtains the following equivalent MIP reformulation of~\eqref{eq:dr-ccp}.
\begin{subequations}\label{eq:joint}
	\begin{align}
		\min\limits_{\vz, \vr, t, \vx} \quad & \vc^\top \vx\label{joint:obj}\\
		\text{s.t.}\quad & \vz \in \{0,1\}^N,\ t \geq 0, \ \vr \geq \bm{0},\ \vx \in \mathcal{X},\label{joint:vars}\\
		& \epsilon\, t \geq \theta \|\vb -\vA^\top \vx\|_* + \frac{1}{N} \sum_{i \in [N]} r^i,\label{joint:conic}\\
		& M^i (1-z^i) \geq t-r^i, \quad i \in [N],\label{joint:bigM1}\\
		& (-\vA\vxi_p^i- \va_p)^\top\vx+ (\vb^\top\vxi_p^i+d_p) + M^i z^i \geq t-r^i, \quad i \in [N]\ p\in[P].\label{joint:bigM2}
	\end{align}
\end{subequations}
where $M^1,\ldots,M^N$ are sufficiently large positive constants. Here, the term $\|\vb - \vA^\top \vx\|_*$ was in the denominator in~\eqref{eq:distance-linear} but is moved by multiplying $t,\vr$ in \eqref{eq:cc-distance-formulation} by $\|\vb - \vA^\top \vx\|_*$ and relabelling the variables accordingly. In fact,~\citet[Proposition 1]{chen2018data} focus on individual chance constraints, but their proof can be extended to provide formulation~\eqref{eq:joint} for joint chance constraints. 

On the other hand, when there exists some $\vx \in \cX\setminus \cX_{\DR}(\cS)$ such that $\vb = \vA^\top \vx$, the constraints \eqref{joint:vars}--\eqref{joint:bigM2} correspond to a \emph{relaxation} of $\cX_{\DR}(\cS)$ for both $\cS\in\{\cS^o,\cS^c\}$. In fact, if $\vx \in \cX\setminus \cX_{\DR}(\cS)$ satisfies $\vb = \vA^\top \vx$, then $\vx$ always satisfies~\eqref{joint:vars}--\eqref{joint:bigM2} together with $t=r^i=0$ and $z^i=1$ for all $i\in[N]$. We next describe the precise relationship between $\cX_{\DR}(\cS)$ for $\cS\in \{\cS^o,\cS^c\}$ and formulation \eqref{eq:joint}.
\begin{lemma}\label{lemma:joint-valid}
We have
\begin{subequations}\label{eq:valid}
\begin{align}
\left\{ \vx : \text{\eqref{joint:vars}--\eqref{joint:bigM2}} \right\} &= \cX_{\DR}(\cS^o) \cup \left\{ \vx \in \cX :~ \begin{aligned}
&\vb = \vA^\top \vx,\\
&d_p \leq \va_p^\top \vx\ \ \text{for some} \ p\in[P]
\end{aligned} \right\}\label{eq:valid-open}\\
&= \cX_{\DR}(\cS^c) \cup \left\{ \vx \in \cX :~ \begin{aligned}
&\vb = \vA^\top \vx,\\
&d_p < \va_p^\top \vx\ \ \text{for some} \ p\in[P]
\end{aligned} \right\}\label{eq:valid-closed}.
\end{align}
\end{subequations}
\end{lemma}
\begin{proof}
Let $\vx \in \cX$ such that $\vb\neq \vA^\top\vx$. Then, by~\citet[Proposition 1]{chen2018data} and \citet[Proposition 1]{xie2018distributionally}, we deduce that $\vx \in \left\{ \vx : \text{\eqref{joint:vars}--\eqref{joint:bigM2}} \right\}$ if and only if $\vx\in \cX_{\DR}(\cS^o)$ and $\vx\in \cX_{\DR}(\cS^c)$.

Now take $\vx \in \cX$ such that $\vb = \vA^\top \vx$. We have already argued that $\vx$ together with $t=r^i=0$ and $z^i=1$ for all $i\in[N]$ satisfies~\eqref{joint:vars}--\eqref{joint:bigM2}. Therefore, to prove that~\eqref{eq:valid-open} and~\eqref{eq:valid-closed} hold, we need to characterize when $\vx$ is contained in $\cX_{\DR}(\cS^o)$ and $\cX_{\DR}(\cS^c)$. 

If $d_p -\va_p^\top \vx > 0$ for all $p\in[P]$, then $\cS^o(\vx)=\bbR^K$ and thus the worst-case probability $\sup_{\bbP \in \cF_N(\theta)} \bbP[\vxi \not\in \cS^o(\vx)]$ is $0$, in which case, $\vx\in \cX_{\DR}(\cS^o)$. On the other hand, if $d_p -\va_p^\top \vx \leq 0$ for some $p\in[P]$, then $\cS^o(\vx)$ is empty, which means that $\sup_{\bbP \in \cF_N(\theta)} \bbP[\vxi \not\in \cS^o(\vx)]=1$. In this case, $\vx\not\in \cX_{\DR}(\cS^o)$. Therefore, the equality in~\eqref{eq:valid-open} holds.

If $d_p -\va_p^\top \vx \geq 0$ for all $p\in[P]$, then $\cS^c(\vx) = \bbR^K$, so as before, $\vx\in \cX_{\DR}(\cS^c)$. If $d_p -\va_p^\top \vx < 0$ for some $p\in[P]$, then $\cS^c(\vx) = \emptyset$ and thus we can similarly argue that $\vx\not\in \cX_{\DR}(\cS^c)$. Hence, the equality in~\eqref{eq:valid-closed} holds, as required.
\ifx\flagJournal\true \qed \fi
\end{proof}

\begin{remark}\label{rem:distance}
\cref{lemma:joint-valid} indicates that the sets $\left\{ \vx : \text{\eqref{joint:vars}--\eqref{joint:bigM2}} \right\}\setminus \cX_{\DR}(\cS^o)$ and $\left\{ \vx : \text{\eqref{joint:vars}--\eqref{joint:bigM2}} \right\}\setminus \cX_{\DR}(\cS^c)$ may potentially be non-empty, in which case, the optimal solution returned by solving~\eqref{eq:joint} may fall into these extraneous sets. \citet[Remark 2]{chen2018data} suggest how to handle this case separately by solving a series of MIPs with strict inequalities. That is, if the optimal solution $\vx^*$ is in the set $\left\{ \vx : \text{\eqref{joint:vars}--\eqref{joint:bigM2}} \right\}\setminus \cX_{\DR}(\cS^o)$, one can solve $2E+1$ variants of~\eqref{eq:joint}, where $E$ is the number of rows in the system $\vb=\vA^\top\vx$, that include exactly one of $2E$ strict inequalities $b_e<(\vA)_e^\top \vx$, $b_e>(\vA)_e^\top \vx$ for $e\in [E]$ and one system of strict inequalities $d_p>\va_p^\top\vx$ for $p\in [P]$. The case when $\vx^*\in \left\{ \vx : \text{\eqref{joint:vars}--\eqref{joint:bigM2}} \right\}\setminus \cX_{\DR}(\cS^c)$ can be similarly dealt with.
\ifx\flagJournal\true \epr \fi
\end{remark}

\begin{remark}\label{rem:joint-valid-RHS-compare}
Note that \cref{lemma:joint-valid} states that the formulation of \eqref{eq:dr-ccp} in equation~\eqref{eq:joint} from previous literature is not exact: there are extraneous parts when $x \in \cX$ is a solution to $A^\top x = b$. This is an artifact of the left-hand side uncertainty. Indeed, in the case of right-hand side uncertainty only, i.e., when $A=0$, $A^\top x = b$ is only solvable when $b=0$ also, but this is the trivial case when the random part of the constraint has been zeroed out, hence we need not consider it. Therefore, in contrast to the RHS uncertainty case discussed in our previous paper \cite{rhs2020}, more effort and care are required when relating the  formulations  for CCP and \eqref{eq:dr-ccp} in the left-hand side case. In particular we need to understand precisely how the extraneous sections are affected, and this is the focus of  \cref{sec:nominal}.
\ifx\flagJournal\true \epr \fi
\end{remark}

\subsection*{Remarks on safety sets}
The most notable structural assumption of~\eqref{eq:safety} is that all $P$ inequalities share the same coefficient matrix $\vA$ as opposed to a more general form that  allow different $\vA$ and $\vb$ matrices across inequalities as the following:
\begin{equation}\label{eq:safety-general}
\cS(\vx) := \left\{ \vxi :~ (\vb_p-\vA_p^\top \vx)^\top \vxi_p + d_p - \va_p^\top \vx \geq 0, \ p \in [P] \right\}. 
\end{equation}
In the RHS uncertainty case, it is possible to have different $\vb_1,\ldots,\vb_p$ instead of the same $\vb$ but $A_p=\bm{0}$ for $p\in[P]$, for which~\citet[Proposition~2]{chen2018data} and~\citet[Corollary~3]{xie2018distributionally} provide almost identical MIP reformulations of $\cX_{\DR}(\cS)$. On the other hand, to the best of our knowledge,  there is no tractable reformulation proposed in the literature for non-identical coefficient matrices $\vA_p, p\in[P]$ in the LHS uncertainty case. 
\begin{remark}\label{distinct-matrices}
The derivation of formulation~\eqref{eq:joint} does not immediately generalize to the case when
$\cS(\vx)$ is given by~\eqref{eq:safety-general}. Although the distance function in~\eqref{eq:distance-linear} can be simply modified with $\vA_p$ and $\vb_p$ for $p\in[P]$, the step of replacing $t\|\vb_p - \vA_p^\top \vx\|_*$ and $r^i\|\vb_p - \vA_p^\top \vx\|_*$ by $t$ and $r^i$ does not go through as before.
\ifx\flagJournal\true \epr \fi
\end{remark}

The form of~\eqref{eq:safety} dictates that we have a fixed matrix $\vA$ for all constraints $p \in [P]$. Despite this restrictive structural form, we next show how the corresponding results can be applied to  more general safety sets of the form~\eqref{eq:safety-general}. 

\begin{remark}\label{rem:same-A}
Given non-identical coefficient matrices $\vA_p$ and vectors $\vb_p$, define the following coefficient matrix and the vector in a lifted space 
\[ \tilde{\vA} := \begin{bmatrix}
\vA_1 & \cdots & \vA_P
\end{bmatrix}, \quad \tilde{\vb} := \begin{bmatrix}
\vb_1^\top & \cdots &\vb_P^\top
\end{bmatrix}^\top. \] 
We also define new random variables $\tilde{\vxi}_p := (\bm{0},\ldots,\bm{0},\vxi_p,\bm{0},\ldots,\bm{0})$ and $\tilde{\vxi} := (\tilde{\vxi}_1,\ldots,\tilde{\vxi}_P)$, i.e., $\tilde{\vxi}_p$ lives in the same space as the original $\vxi = (\vxi_1,\ldots,\vxi_P)$, but all components are set to the zero vector except for $\vxi_p$. 
	Then, letting $\Proj_p$ be the projection operation $\vxi \mapsto \vxi_p$, we can equivalently write $\cS(\vx)$ as
	\[
	\cS(\vx) = \left\{ \vxi = (\vxi_1,\ldots,\vxi_P) :~ \begin{aligned}
	&\exists\, \tilde{\vxi} = (\tilde{\vxi}_1,\ldots,\tilde{\vxi}_P) \text{ s.t. } \vxi_p = \Proj_p(\tilde{\vxi}_p), \ p \in [P],\\
	&\Proj_q(\tilde{\vxi}_p) = \bm{0},\ p \neq q, \ p,q \in [P],\\
	&(\tilde{\vb}-\tilde{\vA}^\top \vx)^\top \tilde{\vxi}_p + d_p - \va_p^\top \vx \geq 0, \ p \in [P]
	\end{aligned} \right\} .
	\]
	The following is an approximation of $\cS(\vx)$ obtained after removing the structural assumption on $\tilde{\vxi}_p$ for $p\in[P]$ by dropping the first two projection constraints in $\cS(\vx)$.
	\[
	\tilde{\cS}(\vx) = \left\{ \tilde{\vxi} = (\tilde{\vxi}_1,\ldots,\tilde{\vxi}_P) :~ (\tilde{\vb}-\tilde{\vA}^\top \vx)^\top \tilde{\vxi}_p + d_p - \va_p^\top \vx \geq 0, \ p \in [P] \right\}.
	\]
	Observe that $\tilde{\cS}(\vx)$ is similar to $\cS(\vx)$, except that it lives in the space of the lifted variables $\tilde{\vxi}$. Importantly, it is of the same form as \eqref{eq:safety}. Note that the ambiguity set $\cF_N(\theta)$ must now consist of distributions over the lifted random variable $\tilde{\vxi}$, rather than $\vxi$. However, since in $\tilde{\cS}(\vx)$  we do not impose that $\Proj_q(\tilde{\vxi}_p) = \bm{0}$ for $p \neq q$ on the support of this random variable $\tilde{\vxi}$, this ambiguity set will be larger than what we originally wish to consider. Therefore, using this lifting given in $\tilde{\cS}(\vx)$ results in a more conservative solution  compared to the optimal solution to~\eqref{eq:dr-ccp}. 
	We  use this technique in the  resource planning application of \cref{sec:resource-planning}. 
	\ifx\flagJournal\true \epr \fi
\end{remark}

\section{Connection with the nominal chance constraint}\label{sec:nominal}

There is a direct relation between~\eqref{eq:dr-ccp} and the traditional sample average approximation formulation:
\begin{equation}\label{eq:saa-ccp}
\min_{\vx} \left\{\vc^\top \vx:~  \vx \in \mathcal{X}, ~ \frac{1}{N} \sum_{i \in [N]} \bm{1}(\vxi_i \not\in \cS(\vx)) \leq \epsilon \right\}. \tag{SAA}
\end{equation}
We formalize this next.  
\begin{remark}\label{rem:SAAconnection} 
When the radius $\theta$ of the Wasserstein ambiguity set $\cF_N(\theta)$ is 0, only the empirical distribution $\bbP_N$ belongs in the ambiguity set, in which case,~\eqref{eq:dr-ccp} coincides with~\eqref{eq:saa-ccp}. In general, \eqref{eq:saa-ccp} is a relaxation of \eqref{eq:dr-ccp} since we have $\cF_N(0)\subseteq \cF_N(\theta)$ for any $\theta\geq0$, i.e.,
\begin{equation}\label{eq:relaxation}
\cX_{\DR}(\cS) \subseteq \cX_{\SAA}(\cS),
\end{equation}
where $ \cX_{\SAA}(\cS)$ denotes the feasible region of~\eqref{eq:saa-ccp}.
Hence, \eqref{eq:saa-ccp} provides a lower bound (for minimization) on the optimum value of~\eqref{eq:dr-ccp}.
\ifx\flagJournal\true \epr \fi
\end{remark}

In the case of RHS uncertainty, this connection between \cref{eq:saa-ccp} and \cref{eq:dr-ccp} has been first observed and explored in our previous work~\citep{rhs2020}. It turns out that this relation is instrumental in improving the MIP formulation of~\eqref{eq:dr-ccp} with LHS uncertainty given in~\eqref{eq:joint} as well. We discuss this in this section. Moreover, this connection
allows us to reduce the extraneous set  
in the feasible region of~\eqref{eq:joint} discussed in \cref{rem:distance}
for the open safety set to $\{ \vx \in \cX : \vb = \vA^\top \vx, \ d_p = \va_p^\top \vx\ \ \text{for some}\ p\in[P]\}$, and remove it completely for the closed safety set.

The relation between~\eqref{eq:dr-ccp} and~\eqref{eq:saa-ccp} described in \cref{rem:SAAconnection} immediately gives rise to an improved formulation. In fact, it turns out that there is a more \emph{direct} correspondence between the MIP reformulation \eqref{eq:saa-reformulation} of~\eqref{eq:saa-ccp} below, often referred to as the \emph{big-$M$ formulation}, and the MIP reformulation~\eqref{eq:joint} of~\eqref{eq:dr-ccp}. As an immediate consequence of \cref{rem:SAAconnection}, we can strengthen the MIP reformulation of~\eqref{eq:dr-ccp}, and further apply existing tools that were developed originally for~\eqref{eq:saa-ccp}. We will elaborate this further in the remainder of this section.  

Suppose that the safety set is given by $\cS(\vx) = \left\{ \vxi :~ s(\vx,\vxi) \geq 0\right\}$ for some continuous function $s(\cdot)$. In this case, \citet{luedtke2010integer,ruspp:02} show that~\eqref{eq:saa-ccp} can be reformulated as the following MIP, known as the big-$M$ formulation:
\begin{subequations}\label{eq:saa-reformulation}
	\begin{align}
	\min\limits_{\vu, \vr, t, \vx} \quad & \vc^\top \vx\label{saa:obj}\\
	\text{s.t.}\quad & \vu \in \{0,1\}^N,\ \vx \in \mathcal{X},\label{saa:vars}\\
	& \sum_{i\in [N]} u^i \leq \lfloor\epsilon N\rfloor, \label{saa:knapsack}\\
	& s(\vx,\vxi^i) + M^i u^i \geq 0, \quad i \in [N],\label{saa:bigM}
	\end{align}
\end{subequations}
where $M^1,\ldots,M^N$ are sufficiently large constants and $u^i,  i \in [N]$  is an indicator variable that is equal to one if $s(\vx,\vxi^i)<0$ and hence scenario $i$ is unsafe. Constraints~\eqref{saa:knapsack} and~\eqref{saa:bigM} are often referred to as the {\it knapsack (or cardinality) constraint} and the {\it big-$M$ constraints}, respectively. Thus, formulation~\eqref{eq:saa-reformulation} provides a relaxation of~\eqref{eq:dr-ccp} when the safety set $\cS(\vx)$ is given by $\cS(\vx) = \left\{ \vxi :~ s(\vx,\vxi) > 0\right\}$. Further strengthenings of the MIP formulation~\eqref{eq:saa-reformulation} via other classes of valid inequalities have been suggested in~\cite{abdi2016mixing-knapsack,Kilinc-Karzan2019joint-sumod,kucukyavuz2012mixing,LKL16,liu2018intersection,luedtke2010integer,luedtke2014branch-and-cut,xie2018quantile,zhao2017joint-knapsack}.

Note that for the closed safety set \eqref{eq:safety-closed}, we  define $s(\cdot)$ as
\begin{equation}\label{eq:saa-safety}
s(\vx,\vxi^i) := \min\limits_{p\in[P]}\left\{(-\vA\vxi_p^i- \va_p)^\top\vx+ (\vb^\top\vxi_p^i+d_p)\right\},
\end{equation}
and thus $s(\vx,\vxi^i)$ in~\eqref{saa:bigM} can be replaced with~\eqref{eq:saa-safety}. Another way of representing~\eqref{saa:bigM} in this case is to expand the minimum term in~\eqref{eq:saa-safety}, thereby obtaining
\begin{equation}\label{eq:saa-safety'}
s_p(\vx,\vxi^i) + M^i z^i \geq 0,\quad \text{where}\quad  s_p(\vx,\vxi^i) := (-\vA\vxi_p^i- \va_p)^\top\vx+ (\vb^\top\vxi_p^i+d_p).
\end{equation}
for $i\in[N]$, $p\in[P]$. As discussed in \cref{rem:SAAconnection},~\eqref{eq:saa-ccp} is a relaxation of~\eqref{eq:dr-ccp}, so inequalities of the form~\eqref{saa:knapsack}--\eqref{saa:bigM} can be added to the MIP formulation \eqref{eq:joint} of~\eqref{eq:dr-ccp}. Including these inequalities in a na\"{i}ve way would introduce a \emph{new} binary variable for each sample $\vxi^i$ and result in two different sets of $N$ binary variables in the formulation. 
Our key observation is that these inequalities can be added \emph{without} introducing additional binary variables, but instead we show in \cref{thm:knapsack-valid-joint} that~\eqref{saa:knapsack}--\eqref{saa:bigM} can simply be added to \eqref{eq:joint} with \emph{$\vz$ simply replacing $\vu$} and \emph{the same big-$M$ constants} without compromising the validity of formulation~\eqref{eq:joint}.
This provides us with the possibility of applying and adapting techniques developed to improve the formulation (and thereby computational tractability) of \eqref{eq:saa-reformulation} to \eqref{eq:joint}. The same observation for strengthening the~\eqref{eq:dr-ccp} formulation in the case of the RHS uncertainty was made in our recent paper~\cite[Theorem 1]{rhs2020}.

Our main result concerns the following MIP formulation for the joint chance constraint of \eqref{eq:dr-ccp} where $\cS(\vx)$ can be either the open or the closed safety set from~\eqref{eq:safety}:
\begin{subequations}\label{eq:joint-knapsack}
	\begin{align}
	\min\limits_{\vz, \vr, t, \vx} \quad & \vc^\top \vx\label{joint-k:obj}\\
	\text{s.t.}\quad & (\vz,\vr,t,\vx)\ \text{satisfies} \ \eqref{joint:vars}\text{--}\eqref{joint:bigM2},\label{joint-k:basic}\\
	& \sum_{i\in [N]} z^i \leq \lfloor \epsilon N \rfloor,\label{joint-k:knapsack}\\
	& (-\vA\vxi_p^i- \va_p)^\top\vx+ (\vb^\top\vxi_p^i+d_p) + M^i z^i \geq 0, \quad i \in [N],\ p\in[P], \label{joint-k:bigM3}
	\end{align}
\end{subequations}
where $M^i, i\in[N]$ are sufficiently large positive constants. 
The only difference between formulation~\eqref{eq:joint-knapsack} and formulation~\eqref{eq:joint} is the additional constraints~\eqref{joint-k:knapsack} and~\eqref{joint-k:bigM3}. 
\cref{thm:knapsack-valid-joint} will show that while~\eqref{joint-k:knapsack} and~\eqref{joint-k:bigM3} do cut off points in the feasible region of~\eqref{eq:joint}, they are nevertheless valid for both $\cX_{\DR}(\cS^o)$ and $\cX_{\DR}(\cS^c)$. In fact, \cref{eq:joint-knapsack} is an exact reformulation of~\eqref{eq:dr-ccp} for the closed safety set $\cS^c$ given in~\eqref{eq:safety}, and gives a tighter relaxation than \eqref{eq:joint} for the open safety set $\cS^o$.

\begin{theorem}\label{thm:knapsack-valid-joint}
The feasible region of~\eqref{eq:joint-knapsack} is characterized as follows:
\begin{align}
&\left\{ \vx \in \cX : \text{\eqref{joint-k:basic}--\eqref{joint-k:bigM3}} \right\} \label{eq:prop1-1-l}\\
&\qquad= \left\{ \vx \in \cX : \text{\eqref{joint:vars}--\eqref{joint:bigM2}} \right\} \setminus \left\{ \vx \in \cX : \begin{aligned}
\ &\vb = \vA^\top \vx,\\
&d_p < \va_p^\top \vx\ \ \text{for some}\ p\in[P]
\end{aligned} \right\}\label{eq:prop1-1}\\
&\qquad= \cX_{\DR}(\cS^o) \cup \left\{ \vx \in \cX : \begin{aligned}
\ &\vb = \vA^\top \vx,\\
&d_p =\va_p^\top \vx\ \ \text{for some}\ p\in[P]
\end{aligned} \right\}\\
&\qquad= \cX_{\DR}(\cS^c).
\end{align}
\end{theorem}
\begin{proof}
We will prove  the equality in~\eqref{eq:prop1-1}. Then the rest will follow from~\eqref{eq:valid}. %

We first show that the set in~\eqref{eq:prop1-1-l} is contained in the set  in~\eqref{eq:prop1-1}. To this end, take a vector $\vx\in\cX$ satisfying~\eqref{joint-k:basic}--\eqref{joint-k:bigM3} with some $\vz,\vr,t$. Then $\vx,\vz,\vr,t$ automatically satisfy \eqref{joint:vars}--\eqref{joint:bigM2}, so it suffices to argue that $\vb \neq \vA^\top \vx$ or $d_p \geq \va_p^\top \vx$ for all $p\in[P]$. Suppose for a contradiction that $\vb = \vA^\top \vx$ and $d_p < \va_p^\top \vx$ for some $p\in[P]$. As $z^i\in\{0,1\}$, it follows from~\eqref{joint:bigM1} and~\eqref{joint:bigM2} that $r^i\geq t$ for all $i\in [N]$. Then we obtain $\epsilon\, t\geq \sum_{i\in[N]}r^i/N\geq t$ from~\eqref{joint:conic}. Since $\epsilon<1$, we must have $t=r^i=0$ for all $i\in[N]$, and then  constraint~\eqref{joint:bigM2} becomes $d_p -\va_p^\top \vx + M^i z^i\geq 0$. This in turn implies that $z^i=1$ for all $i\in[N]$ because $d_p -\va_p^\top \vx<0$, and in particular, $\sum_{i\in[N]}z_i=N$. However, as $\epsilon<1$, $\vz$ violates~\eqref{joint-k:knapsack}, a contradiction. Therefore, $\vb \neq \vA^\top \vx$ or $d_p \geq \va_p^\top \vx$ for all $p\in[P]$, as required.

Next we show that the set in~\eqref{eq:prop1-1} is contained in the set in~\eqref{eq:prop1-1-l}. Let $\vx\in\cX$ satisfy~\eqref{joint:vars}--\eqref{joint:bigM2} with some $\vz,\vr,t$. It suffices to argue that if $\vb \neq \vA^\top \vx$ or $d _p\geq \va_p^\top \vx$ for all $p\in[P]$, then $\vx\in\cX$ satisfies~\eqref{joint-k:basic}--\eqref{joint-k:bigM3} with some $\bar\vz,\bar\vr,\bar t$ (not necessarily the same $\vz,\vr,t$). First, assume that $\vb \neq \vA^\top \vx$. We claim that $\vx,\bar\vz,\vr,t$ satisfy~\eqref{joint-k:basic}--\eqref{joint-k:bigM3} where $\bar\vz\in\{0,1\}^N$ is the vector satisfying $\bar z^i=1$ if and only if $(-\vA\vxi_p^i- \va_p)^\top\vx+ (\vb^\top\vxi_p^i+d_p)<0$ for all $i\in[N]$. Since $M^i$ is sufficiently large so that $(-\vA\vxi_p^i- \va_p)^\top\vx+ (\vb^\top\vxi_p^i+d_p)+M^i\geq0$, the constraints~\eqref{joint-k:bigM3} are satisfied with $\bar\vz$. Moreover, by the choice of $\bar\vz$, $\min\left\{(-\vA\vxi_p^i- \va_p)^\top\vx+ (\vb^\top\vxi_p^i+d_p)+M^i\bar z^i,~M^i(1-\bar z^i)\right\}$ is greater than or equal to $\min\left\{(-\vA\vxi_p^i- \va_p)^\top\vx+ (\vb^\top\vxi_p^i+d_p)+M^iz^i,~M^i(1-z^i)\right\}$ for any $z^i\in\{0,1\}$. That means $\vx,\bar\vz$ satisfy~\eqref{joint:bigM1} and~\eqref{joint:bigM2} because they are already satisfied by $\vx,\vz$. Hence, it remains to argue that $\bar\vz$ satisfies~\eqref{joint-k:knapsack}. Since $\vb \neq \vA^\top \vx$, we have $\|\vb - \vA^\top \vx \|_*>0$ and thus $t>0$ by~\eqref{joint:conic}. We claim that ${r^i}/{t}\geq \bar z^i$ for all $i\in[N]$. When $\bar z^i=1$, ${r^i}/{t}\geq 1=\bar z^i$ holds by~\eqref{joint:bigM1}. We also know that ${r^i}/{t}\geq 0$ as $r^i\geq0$, and in particular, ${r^i}/{t}\geq\bar z^i$ holds when $\bar z^i=0$. As~\eqref{joint:conic} states that $\epsilon N\geq \sum_{i\in[N]}{r^i}/{t}$, it follows that $\epsilon N\geq\sum_{i\in[N]}\bar z^i$. Since $\sum_{i\in[N]}\bar z^i$ takes an integer value, $\bar\vz$ indeed satisfies~\eqref{joint-k:knapsack}. Therefore, $\vx,\bar\vz,\vr,t$ satisfy~\eqref{joint-k:basic}--\eqref{joint-k:bigM3}. Thus, we may assume that $\vb = \vA^\top \vx$ and $d_p\geq\va_p^\top\vx$ for all $p\in[P]$. Then, it is clear that $\vx$ together with $\bar t=\bar r^i=\bar z^i=0$ for $i\in[N]$ satisfies~\eqref{joint-k:basic}--\eqref{joint-k:bigM3}, as required. 
\ifx\flagJournal\true \qed \fi
\end{proof}

\begin{remark}
By \cref{thm:knapsack-valid-joint},~\eqref{eq:joint-knapsack} is an exact reformulation of~\eqref{eq:dr-ccp} when the safety set $\cS(\vx)$ is closed. When $\cS(\vx)$ is open, if~\eqref{eq:joint-knapsack} returns an optimal solution $\vx$ such that $\vb \neq \vA^\top \vx$ or $d_p \neq\va_p^\top \vx$ for all $p\in[P]$, then $\vx$ is an optimal solution to~\eqref{eq:dr-ccp}. However, if~\eqref{eq:joint-knapsack} returns an optimal solution $\vx$ such that $\vb = \vA^\top \vx$ and $d_p =\va_p^\top \vx$ for some $p\in[P]$, then $\vx\not\in \cX_{\DR}(\cS^o)$. Nevertheless, we can deal with this case separately by solving linear programs with strict inequalities as in \citet[Remark 2]{chen2018data} (see also \cref{rem:distance}).
\ifx\flagJournal\true \epr \fi
\end{remark}

\begin{remark}\label{rem:bigM}
\citet[Theorem~2]{xie2018distributionally} shows that the following choice of $M^i$ for $i\in[N]$ is sufficient for the validity of formulations~\eqref{eq:joint} and~\eqref{eq:joint-knapsack}:
\begin{equation}\label{eq:bigMvalue-joint}
M^i=\max_{x\in\cX,p\in[P]}\left\{\left|(-\vA\vxi_p^i- \va_p)^\top\vx+ (\vb^\top\vxi_p^i+d_p)\right|\right\},\quad i\in[N].
\end{equation}
But, when the domain $\cX$ is not bounded, $M^i$ is not necessarily finite; our applications in \cref{sec:experiments} fall into this category. In such cases, instead of \cref{eq:bigMvalue-joint}, we can simply ensure that
\begin{equation}\label{eq:bigMvalue-joint-opt}
M^i \geq \max_{p\in[P]} \left\{\left|(-\vA\vxi_p^i- \va_p)^\top\vx+ (\vb^\top\vxi_p^i+d_p)\right|\right\},\quad i\in[N],
\end{equation}
for at least one optimal solution $\vx$ of \cref{eq:joint-knapsack}, which maintains the validity of the formulation. 
That said, in order to be able to use~\eqref{eq:bigMvalue-joint-opt}, we must understand the structure of the optimal solutions, which can be a nontrivial task on its own. In \cref{sec:experiments}, we will explain how to choose $M^i$ for $i\in[N]$ based on~\eqref{eq:bigMvalue-joint-opt} for the specific applications we consider.
\ifx\flagJournal\true \epr \fi
\end{remark}

\section{Quantile Strengthening}\label{sec:quantile}

Formulation~\eqref{eq:joint-knapsack} is already stronger than~\eqref{eq:joint}. Moreover, we can improve formulation~\eqref{eq:joint-knapsack} even further by exploiting the so-called \emph{mixing substructure} residing in~\crefrange{joint-k:knapsack}{joint-k:bigM3}. In the case of the nominal chance-constrained programs as in~\eqref{eq:saa-reformulation}, analyzing and exploiting the mixing substructure originating from the big-$M$ and the knapsack constraints~\crefrange{saa:knapsack}{saa:bigM} is already a common practice. In particular, the big-$M$ coefficients in front of the binary variables in~\eqref{saa:knapsack} can be significantly reduced based on the assumption that solutions satisfy the knapsack constraint~\eqref{saa:bigM}. We will explain this procedure in \cref{sec:mixing} in detail and refer to it as \emph{quantile strengthening}. \citet{luedtke2010integer} developed this quantile strengthening technique for solving nominal chance-constrained programs with random RHS, and~\citet{luedtke2014branch-and-cut} later applied it to CCPs with random LHS. We can reduce the big-$M$ coefficients in~\eqref{joint-k:bigM3} by applying the same method to~\crefrange{joint-k:knapsack}{joint-k:bigM3}. What is surprising is that the big-$M$ coefficients in~\eqref{joint:bigM2} can also be  reduced using the quantile information, thereby further strengthening formulation~\eqref{eq:joint-knapsack}. %

For distributionally robust chance constraints with random RHS, our previous work \citep{rhs2020} demonstrated how to adapt quantile strengthening to improve the big-$M$ coefficients in~\eqref{joint:bigM2} and provided strong numerical evidence that this coefficient strengthening step has an overwhelmingly positive impact in the overall computation time. In this section, we extend this framework to the~\eqref{eq:dr-ccp} with random LHS setting and discuss how the big-$M$ coefficients in~\eqref{joint:bigM2} can be reduced accordingly. See \cref{rem:mixing-RHS-comparison} for a discussion of the differences in our quantile strengthening procedure for the LHS uncertainty case against our previous paper~\citep{rhs2020}.

The main distinction in the random LHS case, compared to the RHS uncertainty case, is that the coefficients $(-\vA\vxi_p^i-\va_p)$ of the decision variables~$\vx$ in~\eqref{joint:bigM2} change over different scenarios, because $\vA\neq\bf{0}$. When $\vA=\bf{0}$,~\crefrange{joint-k:knapsack}{joint-k:bigM3} naturally give rise to a mixing set with a fixed linear function $(-\va_p)^\top\vx$ for each $p\in [P]$. In contrast, when $\vA\neq\bf{0}$, we construct a mixing set corresponding to $(-\vA\vxi_p^i-\va_p)^\top\vx$ for every pair of $i\in[N]$ and $p\in[P]$.  For this, we rely on an idea of~\citet{luedtke2014branch-and-cut} used for quantile strengthening to solve nominal CCPs with random LHS. 
The distinct feature of our framework is that we consider particular structures stemming from $(-\vA\vxi_p^i-\va_p)^\top\vx$ for $i\in[N]$ and $p\in[P]$ in~\eqref{joint:bigM2} for the sake of reducing the big-$M$ coefficients in~\eqref{joint:bigM2}.

In~\cref{sec:mixing} we describe the construction of the mixing inequalities  as in~\cite{luedtke2014branch-and-cut}, and in~\cref{sec:bigM_reduction} we describe our quantile strengthening procedure for \cref{eq:dr-ccp} with LHS uncertainty.

\subsection{Mixing inequalities}\label{sec:mixing}

Let us consider the following mixing substructure arising from the constraints~\crefrange{joint-k:knapsack}{joint-k:bigM3}:
\begin{equation}
Q:=\left\{ (\vx,\vz)\in \cX \times \{0,1\}^N : \begin{aligned}
\ &s(\vx,\vxi^i) + M^i z^i \geq 0, \quad i \in [N],\\
& \sum_{i\in [N]} z^i \leq \lfloor\epsilon N\rfloor
\end{aligned} \right\}.\label{eq:mixingset}
\end{equation}
We can set $s(\vx,\vxi) := (-\vA\vxi_p- \va_p)^\top\vx+ (\vb^\top\vxi_p+d_p)$ for a fixed $p\in[P]$ so that individual constraints are separately considered, or the set $Q$ can capture the joint constraints by taking  $s(\vx,\vxi):=\min_{p\in[P]}\left\{(-\vA\vxi_p- \va_p)^\top\vx+ (\vb^\top\vxi_p+d_p)\right\}$.%

We will utilize the following procedure to find inequalities of the form $\vmu^\top \vx + \vpi^\top \vz\geq\beta$ that are valid for the mixed-integer set $Q$ in~\eqref{eq:mixingset}. Given a fixed linear function $\vmu^\top\vx$ and a set $\bar{\cX} \supseteq \cX$, we solve the following single scenario subproblem for each scenario $i\in[N]$:
\begin{equation}
\bar h^i(\vmu):=\min\left\{\vmu^\top\vx:\ s(\vx,\vxi^i)\geq0, \ \vx\in \bar{\cX}\right\}.\label{eq:single-subproblem}
\end{equation}
Then, $\vmu^\top\vx\geq \bar h^i(\vmu)$ holds for $(\vx,\vz)\in Q$ with $z^i=0$. Next, we sort  the values $\bar h^i(\vmu)$ for $i\in[N]$ in non-decreasing order. Without loss of generality by a re-indexing if needed, we may assume that $\bar h^N(\vmu) \geq \bar h^{N-1}(\vmu) \geq \cdots \geq \bar h^1(\vmu)$. For ease of notation, we let 
$k:=\lfloor\epsilon N\rfloor.$
Furthermore, note that there must exist $i \in \{N-k,N-k+1,\ldots,N\}$ with $z^i=0$ since $\sum_{i\in [N]} z^i \leq k$ is also enforced in $Q$ and thus the pigeonhole principle applies. So, we deduce that $\vmu^\top\vx \geq \bar h^{N-k}(\vmu)$ because $\bar h^i(\vmu)\geq \bar h^{N-k}(\vmu)$ for all $i\geq N-k$. To summarize, this reasoning shows that $\vmu^\top\vx\geq \bar h^i(\vmu)$ holds if $z^i=0$ and that $\vmu^\top\vx\geq \bar  h^{N-k}(\vmu)$ is satisfied always, in particular, when $z^i=1$ for $i\in [N]$. Hence,
\begin{equation}
\vmu^\top\vx + \left(\bar h^i(\vmu)-\bar h^{N-k}(\vmu)\right)z^i\geq\bar  h^i(\vmu)\label{eq:mixing-base}
\end{equation}
is valid.
Note that the inequalities~\eqref{eq:mixing-base} for $i\leq N-k$ are redundant because $\vmu^\top\vx\geq \bar  h^i(\vmu)$ is implied by $\vmu^\top\vx\geq \bar  h^{N-k}(\vmu)$ if $i\leq N-k$. 
Following this procedure we now have a set of inequalities~\eqref{eq:mixing-base} that share a common linear function $\vmu^\top\vx$ and each one has a distinct integer variable. Therefore, we can   apply the {\it mixing procedure} of~\citet{gunluk2001mixing} (see, also,  {\it star inequalities} by~\citet{atamturk2000mixed}) to the set of inequalities~\eqref{eq:mixing-base} to obtain stronger inequalities. For any $J=\{j_1,\ldots,j_\ell\}$ with $N\geq j_1\geq\cdots\geq j_\ell\geq N-k+1$, the {\it mixing inequality} derived from $J$ and~\eqref{eq:mixing-base} is
\begin{equation}
\vmu^\top\vx + \sum_{i\in[\ell]}\left(\bar h^{j_i}(\vmu)-\bar h^{j_{i+1}}(\vmu)\right)z_{j_i}\geq \bar h^{j_1}(\vmu),\label{eq:mixing}
\end{equation}
where $j_{\ell+1}:=N-k$. 
 Inequalities \eqref{eq:mixing} are sufficient to describe the convex hull of solutions to $(\vx,\vz)\in\bbR^L\times\{0,1\}^N$ satisfying~\eqref{eq:mixing-base}~\cite{atamturk2000mixed,gunluk2001mixing,Kilinc-Karzan2019joint-sumod}. Furthermore, while exponentially many, inequalities~\eqref{eq:mixing} can be separated in $O(N\log N)$ time~\cite{gunluk2001mixing,Kilinc-Karzan2019joint-sumod}.

\begin{remark}\label{remark:relaxed-h}
Inequalities \eqref{eq:mixing} are valid for the set $\bar{\cX}$. If we choose $\bar{\cX} = \cX$, depending on the structure of our original domain and the choice of $s(\cdot)$, computing the value of $\bar h^i(\vmu)$ exactly can be expensive. However, if we take $\bar{\cX} \supseteq \cX$, then inequalities \eqref{eq:mixing} are also valid for $\cX$. Similar to \cref{rem:bigM}, we can also take $\bar{\cX}$ to be a set containing at least one optimal solution to \eqref{eq:joint-knapsack} to also derive valid inequalities for that formulation. In \cref{remark:rsrcp-yield-lower-bounds}, we will follow this computationally more attractive approach  for the probabilistic resource planning application.
\ifx\flagJournal\true \epr \fi
\end{remark}

\subsection{Quantile strengthening via mixing inequalities}\label{sec:bigM_reduction}

We return our attention to formulation~\eqref{eq:joint-knapsack}, which contains a mixing substructure $Q$ of the form \eqref{eq:mixingset} with $s(\vx,\vxi):=\min_{g\in[P]}\left\{(-\vA\vxi_g- \va_g)^\top\vx+ (\vb^\top\vxi_g+d_g)\right\}$. To obtain mixing inequalities \eqref{eq:mixing-base} valid for \eqref{eq:joint-knapsack}, we must choose the linear function $\vmu^\top \vx$. A natural set of candidates for the starting linear function $\vmu^\top \vx$ includes $(-\vA\vxi_p^i- \va_p)^\top \vx$ for $i\in[N]$ and $p\in[P]$. Define $k:=\lfloor \epsilon N\rfloor$ as before. For fixed $i\in[N]$ and $p\in[P]$, let 
\begin{equation}\label{eq:quantiles}
q_p^i:=\text{the} \ (k+1)\text{-th largest value in}\ \left\{\bar h^N(-\vA\vxi_p^i- \va_p),\ldots,\bar h^1(-\vA\vxi_p^i- \va_p)\right\},
\end{equation}
where  for $j\in[N]$
\[ \bar{h}^j(-\vA \vxi_p^i - \va_p) \!=\! \min\! \left\{\! (-\vA\vxi_p^i- \va_p)^\top \vx :\begin{aligned}
&\vx \in \bar{\cX},\\
&(-\vA\vxi_g^j- \va_g)^\top\vx+ (\vb^\top\vxi_g^j+d_g) \geq 0, ~g \in [P]
\end{aligned} \!\right\} \]
as in \eqref{eq:single-subproblem} and $\bar{\cX}$ is a set containing at least one optimal solution to \eqref{eq:joint-knapsack} as in \cref{remark:relaxed-h}. Then, we arrive at the following \emph{basic} mixing inequalities \eqref{eq:mixing-base} that are valid for~\eqref{eq:joint-knapsack}: 
\begin{align}
& (-\vA\vxi_p^i- \va_p)^\top\vx\geq q_p^i,\label{eq:mixing-base-quantile}\\
&(-\vA\vxi_p^i- \va_p)^\top\vx+ (\bar h^i(-\vA\vxi_p^i- \va_p)-q_p^i)z^i\geq \bar h^i(-\vA\vxi_p^i- \va_p).\label{eq:mixing-base-lhs}
\end{align}

\begin{lemma}\label{lemma:quantile-strengthening}
For any $i\in[N]$ and $p\in[P]$, the following inequality is valid for \eqref{eq:joint-knapsack}:
	\begin{equation}\label{eq:mixing-base-relaxed}
	(-\vA\vxi_p^i- \va_p)^\top\vx+ (\vb^\top\vxi_p^i+d_p)+ (-\vb^\top\vxi_p^i-d_p-q_p^i)z^i\geq 0.
	\end{equation}
\end{lemma}
\begin{proof}
From the definition of $\bar h^i(-\vA\vxi_p^i- \va_p)$ above we deduce that 
$$\bar h^i(-\vA\vxi_p^i- \va_p)\geq \min_{\vx}\left\{(-\vA\vxi_p^i- \va_p)^\top\vx:\  (-\vA\vxi_p^i- \va_p)^\top\vx\geq -\vb^\top{\vxi_p^i}-d_p\right\}= -(\vb^\top\vxi_p^i+d_p).$$   
Then, since $\bar h^i(-\vA\vxi_p^i- \va_p)\geq -(\vb^\top\vxi_p^i+d_p)$ and $(z^i-1)\leq 0$ for $z^i\in\{0,1\}$, it follows that $\bar h^i(-\vA\vxi_p^i- \va_p)(z^i-1)\leq -(\vb^\top\vxi_p^i+d_p)(z^i-1)$. So,~\eqref{eq:mixing-base-relaxed} follows from~\eqref{eq:mixing-base-lhs}. 
\ifx\flagJournal\true \qed \fi
\end{proof}
Note that~\eqref{eq:mixing-base-relaxed} is identical to \eqref{joint-k:bigM3} except for a different coefficient in front of the binary variable $z^i$, and \eqref{joint-k:bigM3} itself is quite similar to \eqref{joint:bigM2}. By exploiting this similarity, we can improve formulation~\eqref{eq:joint-knapsack} by reducing the coefficient of $z^i$ in~\eqref{joint:bigM2} to that of \eqref{eq:mixing-base-relaxed}. Thus, our \emph{improved formulation} is:
\begin{subequations}\label{eq:joint-k-reduced}
	\begin{align}
	\min\limits_{\vz, \vr, t, \vx} & \vc^\top \vx\\
	\text{s.t.}~ & (\vz,\vr,t,\vx)\ \text{satisfies \crefrange{joint:vars}{joint:bigM1} and \cref{joint-k:knapsack}}\label{joint-k-reduced:basic}\\
	& (-\vA\vxi_p^i- \va_p)^\top\vx\geq q_p^i, ~i \in [N], p \in [P]\label{joint-k-reduced:quantile}\\
	& (-\vA\vxi_p^i- \va_p)^\top\vx+ (\vb^\top\vxi_p^i+d_p)+ (-\vb^\top\vxi_p^i-d_p-q_p^i)z^i\geq t-r^i, \label{joint-k-reduced:bigM2}\\
	& \qquad \qquad \qquad \qquad \qquad \qquad \qquad \qquad \qquad \qquad \qquad i \in [N], p \in [P].\notag
	\end{align}
\end{subequations}
The validity of the updated inequalities \eqref{joint-k-reduced:bigM2} hinges on the following simple result.
\begin{lemma}\label{lemma:new-bigM}
Suppose that $x,y \in \bbR$, $C,C_1,C_2 \in \bbR_+$ and $z \in \{0,1\}$ satisfies
$C(1-z)\geq y$, %
$x + C_1 z \geq 0$, and %
$x + C_2 z \geq y$.
Then we also have $x + C_1 z \geq y$.
\end{lemma}
\begin{proof}
If $z = 1$, we have $y \leq C(1-z) = 0 \leq x + C_1 z$, and if $z = 0$, we have $x = x + C_1 z = x + C_2 z \geq y$. Thus, in either case, we have $x + C_1 z \geq y$.
\ifx\flagJournal\true \qed \fi
\end{proof}

\begin{theorem}\label{thm:improved-formulation}
	Formulation~\eqref{eq:joint-k-reduced} is a valid reformulation of~\eqref{eq:dr-ccp} where the safety set is given by~\eqref{eq:safety}.
\end{theorem}
\begin{proof}
	By \cref{thm:knapsack-valid-joint}, formulation \eqref{eq:joint-knapsack} is a valid reformulation of~\eqref{eq:dr-ccp}. 
	Thus, it suffices to show that $\cX_1=\cX_2$, where \[\cX_1 := \left\{ (\vz,\vr,t,\vx) : \text{\crefrange{joint-k:basic}{joint-k:bigM3}} \right\},\quad  \cX_2 := \left\{ (\vz,\vr,t,\vx) : \text{\crefrange{joint-k-reduced:basic}{joint-k-reduced:bigM2}} \right\}.\]
	Note that the constraints~\eqref{joint:bigM2} are not explicitly included in $\cX_2$. However, since $M^i \geq -\vb^\top\vxi^i - d_p - q_p^i$ (which we can assume since it is a big-$M$ constant), \eqref{joint:bigM2} is implied by \eqref{joint-k-reduced:bigM2}. Hence, we trivially have $\cX_2 \subseteq \cX_1$. 
	
	In order to prove $\cX_1 \subseteq \cX_2$, we first observe that \crefrange{eq:mixing-base-quantile}{eq:mixing-base-lhs} are simply inequalities \eqref{eq:mixing-base} derived from the mixing substructure \crefrange{joint-k:knapsack}{joint-k:bigM3} using  the function $s(\vx,\vxi)=\min_{p\in[P]}\left\{(-\vA\vxi_p- \va_p)^\top\vx+ (\vb^\top\vxi_p+d_p)\right\}$, thus \eqref{joint-k-reduced:quantile} are valid for $\cX_1$. Finally, we argue that \eqref{joint-k-reduced:bigM2} is valid for $\cX_1$. For every $i \in [N]$ and $p \in [P]$, we obtain from \cref{lemma:quantile-strengthening} and \crefrange{joint:bigM1}{joint:bigM2} that
\begin{subequations}\label{eq:mixing-bigMreduction}
	\begin{align}
	&M^i (1-z^i) \geq t - r^i, \label{eq:mixing-bigMreduction0}\\
	&(-\vA\vxi_p^i- \va_p)^\top\vx+ (\vb^\top\vxi_p^i+d_p)+ (-\vb^\top\vxi_p^i-d_p-q_p^i) z^i\geq 0, \label{eq:mixing-bigMreductionmix}\\
	&(-\vA\vxi_p^i- \va_p)^\top\vx+ (\vb^\top\vxi_p^i+d_p)+ M^i z^i\geq t-r^i. \label{eq:mixing-bigMreductionnon0}
	\end{align}
\end{subequations}
	We then apply \cref{lemma:new-bigM} with $x = (-\vA\vxi_p^i- \va_p)^\top\vx+ (\vb^\top\vxi_p^i+d_p)$, $y = t-r^i$, $C = C_2 = M^i$, $C_1 = -\vb^\top\vxi_p^i-d_p-q_p^i$ to get that \eqref{joint-k-reduced:bigM2} is valid for $\cX_1$.
	\ifx\flagJournal\true \qed \fi
\end{proof}

\begin{remark}
We highlight that, different from the traditional quantile-based strengthening for nominal chance constraints,  the coefficient strengthening proposed in \cref{thm:improved-formulation} is derived from the distinct structure of~\eqref{eq:dr-ccp}, namely the complementary upper bounding constraints \eqref{eq:mixing-bigMreduction0} and \eqref{eq:mixing-bigMreductionnon0} imposed on $t-r^i$ based on the value of $z^i$, combined with the basic mixing inequality \eqref{eq:mixing-bigMreductionmix} that has the same coefficients $-\vA\vxi_p^i- \va_p$ and the same binary variable $z^i$. %
\ifx\flagJournal\true \epr \fi
\end{remark}

\begin{remark}\label{remark:relaxed-quantiles}
The coefficient of $z^i$ in~\eqref{joint-k-reduced:bigM2} is $-\vb^\top\vxi_p^i-d_p-q_p^i$, whereas it is $M^i$ in~\eqref{joint:bigM2}.  Furthermore, $q_p^i$ can be replaced by any lower bound $\beta_p^i$ on $q_p^i$ and the resulting formulation still gives a valid reformulation of~\eqref{eq:dr-ccp}. As long as $M^i\geq -\vb^\top\vxi_p^i-d_p-\beta_p^i$,  the inequality 
\[(-\vA\vxi_p^i- \va_p)^\top\vx+ (\vb^\top\vxi_p^i+d_p)+ (-\vb^\top\vxi_p^i-d_p-\beta_p^i)z^i\geq t-r^i\] 
dominates~\eqref{joint:bigM2}. In practice, the $M^i$ computed na\"{i}vely from \cref{rem:bigM} is much higher than $-\vb^\top \vxi_p^i - d_p - q_p^i$.
\ifx\flagJournal\true \epr \fi
\end{remark}

\begin{remark}\label{remark:warning}
To compute $q_p^i$, we need to evaluate $\bar h^j(-\vA\vxi_p^i- \va_p)$ for $j\in[N]$ which is the optimum value of the single scenario subproblem given in~\eqref{eq:single-subproblem}. Note that we must solve $N^2$ such subproblems. For $s(\vx,\vxi^j)=(-\vA\vxi_p^j- \va_p)^\top\vx+ (\vb^\top\vxi_p^j+d_p)$ in~\eqref{eq:single-subproblem}, this computation becomes
\begin{equation}\label{eq:lower-bounds-indiv}
\bar h^j(-\vA\vxi_p^i- \va_p)=\min\limits_{\vx\in \cX}\left\{(-\vA\vxi_p^i- \va_p)^\top\vx:\ (-\vA\vxi_p^j- \va_p)^\top\vx\geq -\vb^\top\vxi_p^j-d_p\right\}.
\end{equation}
In the optimization problem~\eqref{eq:lower-bounds-indiv}, we can take $\bbR^L$ or $\bbR^L_+$ for a relaxation $\bar \cX$ of $\cX$.  %
But, then the problem in~\eqref{eq:lower-bounds-indiv} becomes trivial and its optimal value is not necessarily finite. In such cases, instead of an individual constraint, we can set  $s(\vx,\vxi^j)=\min_{p\in[P]}\left\{(-\vA\vxi_p^j- \va_p)^\top\vx+ (\vb^\top\vxi_p^j+d_p)\right\}$ in~\eqref{eq:lower-bounds-indiv} so that
\begin{align}
&\bar h^j(-\vA\vxi_p^i- \va_p)\label{eq:lower-bounds-joint}\\
&=\min\limits_{\vx\in \cX}\left\{(-\vA\vxi_p^i- \va_p)^\top\vx: (-\vA\vxi_g^j- \va_g)^\top\vx\geq -\vb^\top\vxi_g^j-d_g,\ g\in[P]\right\}.\nonumber
\end{align}
Although~\eqref{eq:lower-bounds-joint} provides a stronger value than~\eqref{eq:lower-bounds-indiv}, it requires solving a linear program with many constraints even when $\cX=\R^L$ or $\cX=\R_+^L$. In~\cref{sec:cover-pack}, we  study \emph{packing} and \emph{covering} constraints as a special case, where the problems~\eqref{eq:lower-bounds-indiv} are easily solvable. We find that the time taken to compute $q_p^i$ is negligible  for the applications considered in \cref{sec:experiments} even without such covering or packing structure.
\ifx\flagJournal\true \epr \fi
\end{remark}

\begin{remark}\label{rem:mixing-RHS-comparison}

One of the key differences between our previous paper \citep{rhs2020} and the current one lies in the derivation of their strengthenings, i.e.,  the reduction of the big-$M$ coefficients for the DR-CCP formulation in \cref{sec:bigM_reduction} which requires the application of the mixing procedure described in \cref{sec:mixing} in a specific manner. This is more complicated than the procedure used in the right-hand side setting. To be specific, a linear constraint with random right-hand side under $N$ scenarios gives rise to inequalities with the same left-hand side but $N$ different right-hand sides. As they share the same left-hand side, they can be grouped and we can apply the coefficient strengthening developed in~\cite{rhs2020}. However, this is specific to the right-hand side uncertainty case only, because a constraint with random left-hand side under $N$ scenarios gives rise to $N$ inequalities with $N$ different left-hand side terms. In this case, we cannot group these inequalities, and each inequality needs to be dealt with separately. This is definitely a complicating factor that was separately addressed in \cite{luedtke2014branch-and-cut}  for SAA-based chance-constrained programs as well. Our strengthening procedure for the left-hand side uncertainty case of DR-CCP is adapted from \cite{luedtke2014branch-and-cut}.
\ifx\flagJournal\true \epr \fi
\end{remark}

\section{Covering and packing constraints}\label{sec:cover-pack}
Covering and packing problems attracted special attention in the literature  \citep{bicriteria,binary-packing,ZhangJiangShen2018}---their special structures can often be exploited for efficiency. To this end, we now focus on this special case.
Consider constraints of the form
\begin{equation}\label{eq:constraint}
(-\vA\vxi_p- \va_p)^\top\vx>-\vb^\top\vxi_p-d_p \quad\text{and}\quad (-\vA\vxi_p- \va_p)^\top\vx\geq -\vb^\top\vxi_p-d_p,
\end{equation}
where the coefficients $-\vA\vxi_p- \va_p$ of the decision vector $\vx$ and the right-hand side $-\vb^\top\vxi_p-d_p$ have the same sign. In~\eqref{eq:constraint} the strict inequality follows from considering open safety sets. We say that~\eqref{eq:constraint} are \emph{covering} type if $-\vA\vxi_p- \va_p\geq\bf{0}$ and $-\vb^\top\vxi_p-d_p\geq0$, and we say that constraints~\eqref{eq:constraint} are \emph{packing} type if $-\vA\vxi_p- \va_p\leq\bf{0}$ and $-\vb^\top\vxi_p-d_p\leq0$. For example, a \emph{probabilistic portfolio optimization problem} can be defined by a covering constraint; its distributionally robust chance-constrained program formulation is given as follows:
\begin{align}
\min\limits_{\vx} \quad & \vc^\top \vx\notag\\ 
\text{s.t.}\quad & \bbP[\vxi^\top\vx >w ] \geq 1-\epsilon,\quad \forall \bbP \in \cF_N(\theta),\tag{Portfolio}\label{eq:portfolio}\\
& \vx\geq\bm{0},\notag
\end{align}
where $\vxi$ captures the random yields of financial assets; each component encodes the ratio of the end price and the initial price of a financial product (a ratio  greater than~1 implies profit whereas  a  ratio  less than~1 indicates loss). Here, $\vc$ is the cost vector and $w$ denotes a prescribed target return. We may assume that the price never goes down to 0. Then, $\vxi>\bm{0}$ for all $\vxi\in\bbR^K$, and thus~$\vxi^\top\vx >w$ is a covering constraint.

In \cref{sec:bigM_reduction}, we presented a way to improve the value of $M^i$ in~\eqref{joint:bigM2}, which allows us to replace~\eqref{joint:bigM2} by $(-\vA\vxi_p- \va_p)^\top\vx+ (\vb^\top\vxi_p+d_p) + (-\vb^\top\vxi_p-d_p-q_p^i)z^i \geq t-r_i$ where $q_p^i$ is given by~\eqref{eq:quantiles}. Moreover, in \cref{remark:relaxed-quantiles} we argue that $q_p^i$ can be relaxed by any lower bound $\beta_p^i$ on $q_p^i$, especially when the exact evaluation of $q_p^i$ is computationally expensive. Next, we establish that for covering and packing constraints, we can efficiently compute a strong lower bound on $q_p^i$ under a mild assumption.
\begin{lemma}\label{lemma:cov-pack-quantiles}
Suppose that constraints~\eqref{eq:constraint} are in the form of covering or packing. Further, assume that all realizations of $-\vA\vxi_p- \va_p$ have the same support and that $\cX\subseteq \R_+^L$. Then for $i,j\in[N]$ and $p\in[P]$,
\begin{equation}\label{eq:h_j}
\bar h^j_p(-\vA\vxi_p^i- \va_p)\geq\min\left\{(-\vA\vxi_p^i- \va_p)_\ell\frac{(-\vb^\top\vxi_p^j-d_p)}{(-\vA\vxi_p^j- \va_p)_\ell}:\ \ell\in\text{supp}(-\vA\vxi_p^i- \va_p)\right\}
\end{equation}
and $\text{supp}(-\vA\vxi_p^i- \va_p)$ denotes the support of $-\vA\vxi_p^i- \va_p$.
\end{lemma}
\begin{proof}
We consider the case when~\eqref{eq:constraint} are covering type; the packing case can be proved similarly. Since $\cX\subseteq \R_+^L$, it follows from~\eqref{eq:lower-bounds-indiv} that for $i,j\in[N]$ and $p\in[P]$,
\begin{equation}\label{eq:h_j'}
\bar h^j_p(-\vA\vxi_p^i- \va_p)\geq\min\limits_{\vx\geq\bm{0}}\left\{(-\vA\vxi_p^i- \va_p)^\top\vx:(-\vA\vxi_p^j- \va_p)^\top\vx\geq -\vb^\top\vxi_p^j-d_p\right\}.
\end{equation} 
Since $-\vA\vxi_p^i- \va_p$ and $-\vA\vxi_p^j- \va_p$ have the same support, after possibly projecting out some variables in $\vx$, we may assume that $-\vA\vxi_p^i- \va_p>\bm{0}$ and $-\vA\vxi_p^j- \va_p>\bm{0}$. Then, the minimum of the linear program in~\eqref{eq:h_j'} is attained at a vertex of the simplex $\{\vx\in\R_+^L:(-\vA\vxi_p^j- \va_p)^\top\vx=-\vb^\top\vxi_p^j-d_p\}$, thus~\eqref{eq:h_j} follows.
\ifx\flagJournal\true \qed \fi
\end{proof}
Given the lower bounds on $\bar h^j_p(-\vA\vxi_p^i- \va_p)$ for $j\in[N]$ obtained by the closed form in~\eqref{eq:h_j}, the $(N-\lfloor \epsilon N\rfloor)$-th largest one is a lower bound on $q_p^i$, due to~\eqref{eq:quantiles}.

\section{Computational Study}\label{sec:experiments}

We test the effectiveness of our developments on portfolio optimization and probabilistic resource planning problems. We detail the explicit form of these problems, along with instance generation and numerical conclusions in~\cref{sec:portfolio,sec:resource-planning}, respectively. For both problems, we use the $\ell_2$-norm for $\|\cdot\|$ to define the Wasserstein distance~\eqref{eq:wasserstein-dist}.

We conducted all of the experiments on an Intel Core i5 3GHz processor with 6 cores and 32GB memory.  Each experiment was in single-core mode, and five experiments were run in parallel. We enforced a time limit of 3600 seconds on each model.   All solution times are  measured  by C++ in seconds externally from CPLEX.

We used CPLEX 12.9 as the MIP solver. We used CPLEX user-cut callback feature to separate and add cuts from an exponential family. It is well-known that using a user-cut callback function affects various internal CPLEX dynamics (such as dynamic search, aggressiveness of CPLEX presolve and cut generation procedures, etc.). Thus, to make a fair comparison,  
we included an empty user-cut callback function, which does not separate any user cuts, in the implementation of the basic formulation given  by~\citet{chen2018data} and~\citet{xie2018distributionally}. We opted to separate our inequalities only at the root node because we identified in our preliminary tests that separating a large number of inequalities throughout the branch-and-cut tree usually slows down the search process.

\newcommand{\Basic}{\texttt{Basic}\xspace}
\newcommand{\Improved}{\texttt{Improved}\xspace}
\newcommand{\Mixing}{\texttt{Mixing}\xspace}
\newcommand{\Path}{\texttt{Path}\xspace}
\newcommand{\MixingPath}{\texttt{Mixing+Path}\xspace}

We compare the following three formulations:
\begin{description}
	\item[\Basic:] the basic formulation~\eqref{eq:joint} given  by~\citet{chen2018data} and~\citet{xie2018distributionally} where we discuss the big-$M$ computations based on the corresponding problem classes separately below,
	\item[\Improved:] the improved formulation~\eqref{eq:joint-k-reduced},
	\item[\Mixing:]  the improved formulation~\eqref{eq:joint-k-reduced} with the mixing inequalities~\eqref{eq:mixing}.
\end{description}

For each formulation, we recorded the following statistics:
\begin{description}
	\item[Slv(Fnd):] the number of instances solved to optimality within the CPLEX time limit and, in parentheses, the number of instances for which a feasible solution was found, and hence, an upper bound is available.
	\item[Time(Gap):] the average solution time, in seconds, of the instances that were solved to optimality, and, in parentheses, the average of the final optimality gap of the instances that were not solved to optimality within the CPLEX time limit.  The optimality gap is computed as $(UB-LB)/LB*100$ where $UB$ and $LB$ respectively are the objective values of the best feasible solution and the best lower bound value at termination. A `*' in a Time or Gap entry indicates that either no instance was solved to optimality or all instances were solved to optimality within the CPLEX time limit so that there were no instances for measuring the corresponding statistic.
	\item[R.time:] the average time spent at the root node of the branch-and-cut tree over all instances, in seconds.
	\item[R.gap(Fnd):]  the final optimality gap at the root node of the branch-and-cut tree. A `*' entry for gap indicates that no solution was found in any of the 10 instances within the CPLEX time limit,  in parentheses, the number of instances for which a feasible solution was found at the root node, and hence, an upper bound is available. 
\end{description}

\subsection{Portfolio optimization}\label{sec:portfolio}

We consider the distributionally robust chance-constrained programming formulation of a portfolio optimization problem from~\citet{chen2018data} given by~\eqref{eq:portfolio}. The problem is to find a minimum cost portfolio investment $\vx$ into $K$ assets with random returns $\vxi=(\xi_1,\ldots,\xi_K)^\top\in\R_+^K$ while achieving a prescribed target return with probability at least $1-\epsilon$. Problem~\eqref{eq:portfolio} admits the following MIP reformulation:
\begin{subequations}\label{eq:portfolio-re}
	\begin{align}
	\min\limits_{\vz, \vr, t, \vx} \quad & \vc^\top \vx\label{portfolio:obj}\\
	\text{s.t.}\quad & \vz \in \{0,1\}^N,\ t \geq 0, \ \vr \geq \bm{0},\ \vx \geq \bm{0},\label{portfolio:vars}\\
	& \epsilon\, t \geq \theta \|\vx\|_* + \frac{1}{N} \sum_{i \in [N]} r^i,\label{portfolio:conic}\\
	& M^i (1-z^i) \geq t-r^i, \quad i \in [N],\label{portfolio:bigM1}\\
	& \vx^\top\vxi^i-w + M^i z^i \geq t-r^i, \quad i \in [N],\label{portfolio:bigM2}\\
	&\sum_{i\in[N]}z^i\leq \lfloor \epsilon N\rfloor.\label{portfolio:knapsack}
	\end{align}
\end{subequations}
In fact, $\vx =\bm{0}$ with $(\vz,\vr,t)=(\bm{1},\bm{0},0)$ satisfies~\eqref{portfolio:vars}--\eqref{portfolio:bigM2}. Then $\vx =\bm{0}$ with $(\vz,\vr,t)=(\bm{1},\bm{0},0)$ would be an optimal solution if~\eqref{portfolio:knapsack} were not present. Hence,~\eqref{portfolio:knapsack} is necessary, and by \cref{thm:knapsack-valid-joint},~\eqref{eq:portfolio-re} is an exact reformulation of~\eqref{eq:portfolio}. 

Adapting our formulation~\eqref{eq:joint-k-reduced} to model~\eqref{eq:portfolio}, we obtain another formulation that is the same as~\eqref{eq:portfolio-re} except that~\eqref{portfolio:bigM2} is replaced with 
\begin{equation}\label{portfolio:bigM2'}
\vx^\top\vxi^i-w + (w-q^i)z^i \geq t-r^i,\quad i\in[N],
\end{equation}
where $q^i$ is defined as in~\eqref{eq:quantiles}. As $\vxi^\top\vx >w$ is a covering constraint, we can compute a lower bound $q^i$ based on~\eqref{eq:h_j} in \cref{lemma:cov-pack-quantiles}. We next discuss how to select valid big-$M$ values in~\eqref{eq:portfolio-re}. %

\begin{remark}\label{remark:portfolio-M}
For~\eqref{eq:portfolio-re}, the domain of $\vx$ is not bounded, and hence $M^i$ given by~\eqref{eq:bigMvalue-joint} is not bounded. Then, as discussed in \cref{rem:bigM}, for some optimal $\vx$ to \cref{eq:portfolio-re}, we can choose $M^i\geq |\vx^\top \vxi^i-w|$ for each $i \in [N]$. Let $\vx$ be an optimal solution to~\eqref{eq:portfolio-re}.
First, since $\vx^\top \vxi^i\geq0$, it follows that $(\vx^\top \vxi^i-w)+w=\vx^\top \vxi^i\geq0$, so $-(\vx^\top \vxi^i-w)\leq w$ for all $i\in[N]$. Let $J\subseteq [N]$ denote the set of scenarios $j$ such that $\vx^\top \vxi^j-w\geq 0$. Then, $J$ is nonempty because $\vx$ satisfies the nominal chance constraint with nonzero probability. If $\vx^\top \vxi^j-w>0$ for all $j\in[J]$, one can scale down $\vx$ by a factor of some $\delta\in(0,1)$ such that $\delta\vx^\top \vxi^j-w\geq0$ for $j\in[J]$, thereby satisfying the same set of scenarios but obtaining a better solution. So, we may assume that $\vx^\top \vxi^j=w$ for some $j\in J$. Let $\xi_{\max}$ and $\xi_{\min}$ be the maximum and the minimum coordinate values of $\vxi$. Then, for $j\in[N]$, $\vx^\top \vxi^i\leq \xi_{\max} \vx^\top\bm{1}\leq {\xi_{\max}}\vx^\top \vxi^j/{\xi_{\min}}=w\cdot{\xi_{\max}}/{\xi_{\min}}$,
implying that $(\xi_{\max}/\xi_{\min}-1)w \geq (\vx^\top \vxi^i-w)$ holds for all $i\in[N]$. Thus, it is sufficient to set
\begin{equation}\label{portfolio:M}
M^i=\max\{w,\ \left({\xi_{\max}}/{\xi_{\min}}-1\right)w  \},\quad i\in[N].
\end{equation}
\ifx\flagJournal\true \epr \fi
\end{remark}

\subsubsection{Instance Generation}

We follow the same instance generation scheme of~\citet{chen2018data} (and hence that of~\citet{Xie18}). We set $K=50$, $w=1$, and the cost coefficients $c_i$, for $i\in[50]$, %
are chosen uniformly at random from $\{1,\ldots,100\}$. As mentioned in \cref{sec:cover-pack}, each $\xi_i$ indicates the ratio of the end price and the initial price so that $\vxi$ always remains positive. For our experiments, we generate each coordinate of $\vxi$  uniformly at random from $[0.8,1.5]$. Based on \cref{remark:portfolio-M} and~\eqref{portfolio:M}, we set $M^i=1$ for all $i \in [N]$. As we use the $\ell_2$-norm for Wasserstein ambiguity sets, reformulations~\eqref{eq:joint} and~\eqref{eq:joint-k-reduced} become mixed-integer second-order cone programs. We test a set of values for the Wasserstein radius $\theta$ and risk tolerance $\epsilon$; we choose $\theta\in\{0.05,0.1,0.2\}$ and $\epsilon\in\{0.05,0.1\}$. For each problem parameter combination, we generate 10 random instances and report the average statistics.

\subsubsection{Performance Analysis} \label{subsec:portfolio-results}

Our experiments  with $N\in\{500,1000\}$ scenarios are summarized in \cref{tab:port-full}.
Note that these correspond to much larger number of scenarios than $N\in\{100,110,\ldots,200\}$ considered previously in~\cite{chen2018data}. For completeness, we present experiments on $N\in\{100,300\}$ scenarios and a brief discussion in Appendix~\ref{sec:portfolio_small}.

\begin{sidewaystable}
	\centering
	\caption{Results for portfolio optimization}
	\label{tab:port-full}
	\begin{tabular}{ll|rrrr|rrrr}
		\toprule
		$N$ & $\theta$ & \multicolumn{4}{c}{\texttt{Basic}} & \multicolumn{4}{c}{\texttt{Improved}} \\
		&       &       Slv(Fnd) &      Time(Gap) & R.time & R.gap(Fnd) &          Slv(Fnd) &       Time(Gap) & R.time & R.gap(Fnd) \\
		\midrule
		\multirow{10}{*}{500} & 0.001 &          0(10) &       *(18.85) &   1.28 &  22.49(10) &             0(10) &        *(14.29) &   0.29 &  17.99(10) \\
		& 0.020 &          4(10) &  330.78(17.97) &   1.56 &  19.46(10) &             0(10) &        *(12.59) &   2.40 &  34.15(10) \\
		& 0.040 &          4(10) &  283.29(17.69) &   1.83 &  22.26(10) &             4(10) &      1.89(9.09) &   2.28 &  16.80(10) \\
		& 0.060 &          2(10) &  391.52(14.24) &   5.02 &   30.27(9) &             5(10) &    459.36(8.62) &   4.54 &  24.89(10) \\
		& 0.080 &          1(10) &  646.29(10.94) &   5.36 &   28.82(8) &             7(10) &    820.63(3.50) &   4.98 &  22.68(10) \\
		& 0.100 &          1(10) &  1925.59(9.65) &   5.45 &   27.79(9) &             8(10) &    445.88(2.45) &   5.59 &  13.04(10) \\
		& 0.120 &          3(10) &  1027.24(8.42) &   6.79 &  32.11(10) &            10(10) &       168.85(*) &   5.36 &    9.72(9) \\
		& 0.140 &          5(10) &   463.89(7.41) &   6.56 &  33.94(10) &            10(10) &       355.98(*) &   5.69 &   10.55(9) \\
		& 0.160 &          7(10) &   193.68(5.55) &   6.76 &  28.42(10) &            10(10) &         9.89(*) &   5.65 &   8.01(10) \\
		& 0.180 &          9(10) &   571.28(2.59) &   6.21 &  23.23(10) &            10(10) &         6.18(*) &   4.24 &   3.87(10) \\
		\cline{1-10}
		\multirow{10}{*}{1000} & 0.001 &          0(10) &       *(23.87) &   2.27 &  26.55(10) &             0(10) &        *(16.37) &   0.77 &  17.89(10) \\
		& 0.020 &          0(10) &       *(23.73) &   5.09 &   26.90(9) &              0(9) &        *(17.51) &   3.53 &   23.45(9) \\
		& 0.040 &           1(5) &   10.65(18.03) &   4.37 &   24.32(5) &              3(5) &  1173.12(12.20) &   3.53 &    7.40(4) \\
		& 0.060 &           0(4) &       *(24.83) &   9.43 &   32.80(1) &              3(5) &  1445.78(14.05) &  10.39 &       *(0) \\
		& 0.080 &           0(9) &       *(24.19) &  12.70 &       *(0) &              5(9) &   164.61(10.11) &  11.07 &       *(0) \\
		& 0.100 &           0(9) &       *(19.45) &  13.58 &       *(0) &             8(10) &   242.90(11.46) &  11.87 &       *(0) \\
		& 0.120 &          0(10) &       *(17.16) &  14.58 &       *(0) &             7(10) &     44.15(9.33) &  13.01 &       *(0) \\
		& 0.140 &          0(10) &       *(15.67) &  15.62 &       *(0) &             8(10) &   296.48(10.05) &  13.72 &       *(0) \\
		& 0.160 &           0(9) &       *(13.15) &  16.50 &       *(0) &             8(10) &     80.18(8.43) &  13.36 &       *(0) \\
		& 0.180 &          0(10) &       *(12.03) &  17.25 &       *(0) &             9(10) &    114.21(3.46) &  13.74 &       *(0) \\
		\bottomrule
	\end{tabular}
\end{sidewaystable}

For both $N\in\{500, 1000\}$, mixing inequalities are very rarely separated when $\theta=0.02$, and they are never separated for large $\theta>0.02$. When $\theta>0.02$, since mixing inequalities are  never separated, the performances of \Mixing and \Improved are almost identical in terms of all of the statistics including root node statistics. The non-separation of mixing inequalities for large $\theta$ follows from the fact that the nominal region $\cX_{\SAA}(\cS)$ (and consequently the resulting mixing inequalities) is a worse approximation for the distributionally robust region $\cX_{\DR}(\cS)$ when $\theta$ gets larger. The same phenomenon was also observed in \cite{rhs2020} fo DR-CCP with RHS uncertainty. Thus, we report  the relevant statistics for \Mixing only for $\theta \leq 0.02$  in \cref{tab:mixing} in \cref{app:mixing}.  

We observe that when the radius $\theta$ is small, the resulting problems are much harder to solve. Such difficulty of the problems for small $\theta$ was also reported by~\citet{rhs2020} for DR-CCP with RHS uncertainty. For example, for $\theta=0.001$, none of the models is able to solve  any one of the ten instances for $N=500$ or $N=1000$ within the time limit of 3600 seconds. Despite this, we observe that in terms of the average final optimality gap for $\theta=0.001$ and $N=500$ ($N=1000$), \Mixing is the best with 6.10\% gap (8.79\%), followed by \Improved with 14.29\% (16.37\%), and finally \Basic with 18.85\% (23.87\%). 
In the case of $\theta=0.001$, it is noteworthy to point out that the average number of mixing inequalities separated is still relatively small; 89.6 in the case of $N=500$ and 271.5 for $N=1000$. 
However, for $\theta=0.001$ and $N=1000$, comparing \Improved and \Mixing, we note that the mixing inequalities 
improve the average root gap from 17.89\% to 16.43\%. This may appear to be a modest reduction, but surprisingly, it resulted in a reduction in the final optimality gap from 16.37\% to 8.79\% on average. %
Overall, these results highlight the positive computational impact of our developments in \Improved and \Mixing for small $\theta$.

As for the other $\theta$ values, we observe that \Improved consistently outperforms \Basic in terms of the number of instances solved for  all $N$ and $\theta$ values. This is particularly striking for $N=1000$. In this case, \Basic is unable to solve (with the exception of one instance out of ten for $\theta=0.04$) any of the ten randomly generated problem instances for any of the $\theta$ values within the time limit of 3600 seconds. In contrast, for all of the $\theta$ values greater than or equal to $0.1$, \Improved solves at least 7 out of 10 random instances within an average of less than 300 seconds. For the instances that were unsolved for $N=1000$ and $\theta\ge0.04$, the reported average final gaps for \Basic range between 12\% to 24.8\%, whereas the same range for \Improved is 3.5\% to 14\%. 
It may appear that for $N=500$ and $\theta\in\{0.04,0.06\}$, overall solution time of \Improved  is longer than \Basic solution times, but this is due to the fact that we are able to solve more instances with \Improved within the time limit (5 and 7 versus 2 and 1). For some  instances, even finding a feasible solution within the time limit is a challenge for both of the formulations, in particular for $N=1000$ and $\theta\in\{0.04,0.06\}$. 
Finally, observe that the solution time at the root node for  \Basic and \Improved are very similar, but the
 root node gap   of \Improved is  better than \Basic in most cases. This difference is more pronounced for the instances with $N=500$ and large $\theta$. The large improvement in the root gap for these instances translates into much faster overall solution times.

\subsection{Probabilistic resource planning}\label{sec:resource-planning}

We consider a probabilistic resource planning problem studied by~\citet{luedtke2014branch-and-cut} in the context of solving~\eqref{eq:saa-ccp}. Given a set of resources %
and a set of customer groups, %
the problem is to decide the quantity of each resource with the minimum cost to satisfy customer demands, i.e., 
\begin{equation}\label{eq:resource-panning}
\min_{\vx\in\bbR_+^D,\ \vy\in\bbR_+^{DP}}\left\{\vc^\top\vx:\ \begin{array}{ll}
\sum_{p\in [P]}y_{dp}\leq \rho_dx_d,&d\in[D]\\
\sum_{d\in [D]}\mu_{dp}y_{dp}\geq\lambda_p,& p\in[P]
\end{array}\right\},\tag{RSRC-Plan}
\end{equation}
where $D$ is the number of resources and $P$ is the number of customer types, $c_d$ is the unit production cost of resource $d\in[D]$ and $\rho_d\in(0,1]$ represents the random yield of resource $d$ (e.g., the fraction of planned production that is available), $\lambda_p$ denotes the random demand of customer group $p\in[P]$, %
$\mu_{dp}$ represents the random service rate of resource $d$ for customer group $p$. %
 Here, $x_d$ is the variable for the quantity of resource $d$ to be produced and $y_{dp}$ is the variable for the amount of resource $d$ allocated to customer group $p$. The constraints $\sum_{p\in [P]}y_{dp}\leq \rho_dx_d$ for $d\in[D]$ in~\eqref{eq:resource-panning} are \emph{resource assignment constraints}, and $\sum_{d\in [D]}\mu_{dp}y_{dp}\geq\lambda_p$ for $p\in[P]$ are  \emph{demand satisfaction constraints}. Let $(\vrho^i,\vmu^i,\vlambda^i)\in\bbR_+^D\times\bbR_+^{DP}\times\bbR_+^P$ be the realization of the random parameters under scenario $i\in[N]$. Then the DR-CCP formulation of~\eqref{eq:resource-panning} is given by %
\begin{subequations}\label{eq:rsrcp-re}
	\begin{align}
	\min\limits_{\vz, \vr, t, \vx,\vy} \quad & \vc^\top \vx\label{rsrcp:obj}\\
	\text{s.t.}\quad & \vz \in \{0,1\}^N,\ t \geq 0, \ \vr \geq \bm{0},\ \vx \geq \bm{0},\ \vy \geq \bm{0},\label{rsrcp:vars}\\
	& \epsilon\, t \geq \theta \left\|(\vx,\vy,\bm{1})^\top\right\|_* + \frac{1}{N} \sum_{i \in [N]} r^i,\label{rsrcp:conic}\\
	& M^i(1-z^i) \geq t-r^i, \quad i \in [N],\label{rsrcp:bigM1}\\
	& \rho_d^ix_d-  \sum_{p\in [P]}y_{dp} + M^iz^i \geq t-r^i, \quad i \in [N],\ d\in[D],\label{rsrcp:bigM2-d}\\
	& \sum_{d\in [D]}\mu_{dp}^iy_{dp}-\lambda_p^i + M^iz^i \geq t-r^i, \quad i \in [N],\ p\in[P]. \label{rsrcp:bigM2-p}%
	\end{align}
\end{subequations}

\subsubsection{Instance generation and big-$M$ computation}
We test instances with $D=10$, $P=20$, and $\epsilon=0.1$. For the cost vector $\vc$ and the random parameters $(\vrho,\vmu,\vlambda)$, we use the same setting of~\citet[Section 3.3]{luedtke2014branch-and-cut} (further details of instance generation can be found in~\citet{luedtke2014branch-and-cut-supplement}). This instance generation scheme ensures that each sample data $(\vrho^i,\vmu^i,\vlambda^i)$ is nonnegative almost surely. We empirically found that the problem becomes infeasible when $\theta$ gets above 0.01, so we test 10 different values $\{0.0001,0.001,0.002,\ldots,0.009\}$ for $\theta$.

Since the domain $\cX$ of~\eqref{eq:rsrcp-re} is not bounded, we need to choose a value for $M^i$ based on~\eqref{eq:bigMvalue-joint-opt}, i.e., for some optimal $(\vx,\vy)$ to \cref{eq:rsrcp-re}, set $M^i$ to be greater than or equal to
\begin{equation}\label{rsrcp:M-lb}
\max_{d\in[D]}\bigg\{\bigg|\rho_d^ix_d-  \sum_{p\in [P]}y_{dp}\bigg|\bigg\}\ \text{and} \ \max_{p\in[P]}\bigg\{\bigg|\sum_{d\in [D]}\mu_{dp}^iy_{dp}-\lambda_p^i\bigg|\bigg\}.
\end{equation}
Using the nonnegativity of data $(\vrho^i,\vmu^i,\vlambda^i)$, in \cref{remark:rsrcp-M} in~\cref{app:big-M-resource-planning}, we provide an upper bound on~\eqref{rsrcp:M-lb}, thereby providing a value for $M^i$.

Notice that the demand  constraints in~\eqref{eq:resource-panning} are covering type, so we can improve~\eqref{rsrcp:bigM2-p} by reducing $M^i$ based on \cref{lemma:cov-pack-quantiles}. However, the resource assignment constraints are neither covering nor packing type, hence  we cannot apply \cref{lemma:cov-pack-quantiles} to compute the reduced coefficient for~\eqref{rsrcp:bigM2-d}. So, %
in~\cref{remark:rsrcp-yield-lower-bounds} we describe our reduced coefficient computation for~\eqref{rsrcp:bigM2-d} based on~\eqref{eq:lower-bounds-joint}. %

\subsubsection{Performance Analysis} \label{subsec:resource-results}

We summarize our experiments  with $N\in\{100,300\}$ scenarios in \cref{tab:RP-full}.
Note that the resource planning problems with LHS uncertainty are significantly more difficult than  the portfolio optimization problems, thus the number of scenarios $N$ we can scale to were much smaller than in \cref{sec:portfolio}.

\begin{sidewaystable}
	\centering
	\caption{Results for resource planning}
	\label{tab:RP-full}
	\begin{tabular}{ll|rrrr|rrrr}
		\toprule
		&        & \multicolumn{4}{c}{\texttt{Basic}} & \multicolumn{4}{c}{\texttt{Improved}} \\
		$N$ & $\theta$ &       Slv(Fnd) &       Time(Gap) & R.time & R.gap(Fnd) &          Slv(Fnd) &       Time(Gap) & R.time & R.gap(Fnd) \\
		\midrule
		\multirow{10}{*}{100} & 0.0001 &          4(10) &  1711.87(54.86) &   5.55 &  100.00(4) &             6(10) &    979.56(4.64) &   7.82 &       *(0) \\
		& 0.0010 &          0(10) &        *(77.39) &   8.29 &       *(0) &             2(10) &   2226.70(9.03) &  10.61 &       *(0) \\
		& 0.0020 &          0(10) &        *(77.05) &   8.72 &       *(0) &             6(10) &   2082.06(8.11) &  11.09 &       *(0) \\
		& 0.0030 &          1(10) &  2497.16(61.80) &   8.58 &       *(0) &             8(10) &   1142.60(4.53) &  11.36 &       *(0) \\
		& 0.0040 &          2(10) &  1398.14(62.78) &   8.49 &       *(0) &             8(10) &  1184.37(10.26) &  11.44 &       *(0) \\
		& 0.0050 &          1(10) &  2168.77(62.35) &   8.75 &       *(0) &             6(10) &  1521.19(14.51) &  11.74 &       *(0) \\
		& 0.0060 &          3(10) &  2198.37(60.46) &   7.78 &       *(0) &             8(10) &  1058.29(12.99) &  11.52 &       *(0) \\
		& 0.0070 &          3(10) &  2266.62(61.80) &   7.62 &       *(0) &            10(10) &      2107.87(*) &  11.56 &       *(0) \\
		& 0.0080 &          4(10) &  2699.47(54.28) &   7.73 &       *(0) &            10(10) &       978.87(*) &  11.26 &       *(0) \\
		& 0.0090 &          7(10) &  2456.14(45.00) &   7.40 &       *(0) &            10(10) &       792.77(*) &  11.35 &       *(0) \\
		\cline{1-10}
		\multirow{10}{*}{300} & 0.0001 &          0(10) &        *(91.92) &  25.35 &       *(0) &             0(10) &         *(7.68) &  17.94 &   11.25(6) \\
		& 0.0010 &          0(10) &        *(94.81) &  33.05 &       *(0) &             0(10) &        *(13.22) &  44.19 &       *(0) \\
		& 0.0020 &          0(10) &        *(93.35) &  31.62 &       *(0) &             0(10) &        *(14.07) &  56.40 &       *(0) \\
		& 0.0030 &          0(10) &        *(94.17) &  28.80 &       *(0) &             0(10) &        *(13.54) &  63.93 &       *(0) \\
		& 0.0040 &          0(10) &        *(93.64) &  29.30 &       *(0) &             0(10) &        *(13.78) &  63.74 &       *(0) \\
		& 0.0050 &          0(10) &        *(94.47) &  28.45 &       *(0) &             0(10) &        *(13.87) &  64.74 &       *(0) \\
		& 0.0060 &          0(10) &        *(94.24) &  26.96 &       *(0) &             3(10) &  2283.06(11.67) &  67.35 &       *(0) \\
		& 0.0070 &          0(10) &        *(96.06) &  28.50 &       *(0) &             4(10) &  1815.24(11.08) &  72.66 &       *(0) \\
		& 0.0080 &          0(10) &        *(96.77) &  27.20 &       *(0) &             7(10) &  1729.98(10.80) &  74.28 &       *(0) \\
		& 0.0090 &          0(10) &        *(97.08) &  22.89 &       *(0) &            10(10) &       960.34(*) &  73.41 &       *(0) \\
		\bottomrule
		\end{tabular}
\end{sidewaystable}

We continue to see that when the radius $\theta$ is small, the resulting problems are much harder to solve. 
For example, for $\theta\in\{0.001,0.002\}$, \Basic  is not able to solve  any one of the ten instances for $N=100$ within the time limit. That said, for the really small radius of $\theta=0.0001$, the instances are slightly easier with more instances solved to optimality than $\theta=0.001$ for all models.
For $N=100$, as $\theta$ increases, more instances are solved by \Basic, however, even for the largest $\theta$, i.e., $\theta=0.009$, there are three instances for which \Basic is not able to find an optimal solution. In contrast \Improved is able to solve all instances to optimality for $\theta\in\{0.007,0.008,0.009\}$. 
 Furthermore, for $N=300$, \Basic is not able to solve any of the instances for \emph{any} $\theta$. These instances are simply intractable for \Basic, which terminates with over 90\% optimality gap in all test cases. In contrast, for the largest $\theta$ tested, \Improved finds an optimal solution to all ten instances well within the time limit.

Comparing  the quality of the solutions  at termination, we observe that the optimality gaps for \Basic are extremely large in these instances, ranging from 45\% for $N=100, \theta=0.009$ to 97.08\% for $N=300, \theta=0.009$. In contrast, the optimality gaps for \Improved range from 0\% for various settings including   $N=100, \theta\in\{0.007,0.008,0.009\}$ and $N=300, \theta=0.009$ to at most 14.51\% for $N=100, \theta=0.005$.

It is interesting to note that in most cases, an integer feasible solution is not found at the root node in both \Basic and \Improved, so the root gap information is not available, except \Basic is able to find a feasible solution for four instances for $N=100, \theta=0.0001$, albeit with 100\% optimality gap.
This observation is reversed for $N=300, \theta=0.0001$, when \Basic is unable to report a root gap for any instance, whereas \Improved is able to report an average gap of 11.25\% for six instances. 

A few comments are in order for the performance of \Mixing. Once again, we only report these results for $\theta \leq 0.001$ in \cref{app:mixing}, since we observed that no mixing inequalities are separated for $\theta > 0.001$ for any $N\in\{100,300\}$. That said, \Mixing is quite effective for $\theta=0.0001$. For $N=100, \theta=0.0001$, an average of 687.3  mixing inequalities are separated, and \Mixing is able to solve nine instances to optimality (three more than \Improved), and in smaller average solution time (800 seconds versus 979 seconds). The effectiveness of \Mixing decreases when $\theta=0.001$: in this case, only an average of 12.5 mixing inequalities are separated when $N=100$. Indeed, in this case, \Mixing solves one fewer instance to optimality than \Improved and there is only a moderate decrease in the optimality gap (7.99\% for \Mixing versus 9.03\% for \Improved). More mixing cuts are separated on average  for $N=300$: 4337.9 and 223.5, respectively, for $\theta=0.0001$ and $\theta=0.001$. Despite this, these instances are still unsolvable within the time limit. Nevertheless, there is a moderate decrease in the final gap from 7.68\% to 7.55\% for $\theta=0.0001$ and from 13.22\% to 11.88\% for $\theta=0.001$. Finally, with respect to root gaps, in contrast to \Improved, the only interesting statistic for \Mixing is that  $N=300, \theta=0.0001$, \Mixing achieves a smaller root gap of  10.28\%, on average, but over fewer instances that solve to optimality (five instead of  six) than \Improved. %

In summary, we observe that our proposed \Improved formulation drastically increases our ability to obtain high-quality solutions to  \eqref{eq:dr-ccp}.  \Mixing provides additional improvement for cases when $\theta$ is small.

\section*{Acknowledgments}
This research is supported, in part, by ONR grant N00014-19-1-2321, by the Institute for Basic Science (IBS-R029-C1, IBS-R029-Y2), Award N660011824020 from the DARPA Lagrange Program and NSF Award 1740707.

\bibliographystyle{abbrvnat}
\bibliography{mybibfile}

\newpage

\appendix
\section{Additional Results for Portfolio Optimization}\label{sec:portfolio_small}

In this section, we present additional results for portfolio optimization problems from Section~\ref{subsec:portfolio-results} for $N\in\{100, 300\}$. The results in Tables \ref{tab:port-full-N100} and \ref{tab:port-full-N300} demonstrate once again the improved performance of our approach for \emph{all} parameter regimes, including small $N$ and large $\theta$. Note that while the basic formulation does not have trouble solving portfolio instances with $N=100$ (and also for $N=300$ and $\theta \geq 0.4$), we still observe that the improved formulation solves noticeably more quickly for these parameter regimes.
 
	\begin{table}[h!t]
	\centering
	\caption{Results for portfolio for $N=100$}
	\label{tab:port-full-N100}
	\begin{tabular}{l|rrrr|rrrr}
		\toprule
		{} & \multicolumn{4}{c}{\texttt{Basic}} & \multicolumn{4}{c}{\texttt{Improved}} \\
		{$\theta$} &       Slv(Fnd) & Time(Gap) & R.time & (Fnd)R.gap &          Slv(Fnd) & Time(Gap) & R.time & (Fnd)R.gap \\
		\midrule
		0.001  &         10(10) &   0.75(*) &   0.15 &  (10)18.60 &            10(10) &   0.29(*) &   0.04 &  (10)16.25 \\
		0.020  &         10(10) &   0.61(*) &   0.40 &  (10)18.92 &            10(10) &   0.59(*) &   0.11 &  (10)29.14 \\
		0.040  &         10(10) &   0.94(*) &   0.74 &  (10)19.72 &            10(10) &   0.56(*) &   0.30 &  (10)14.92 \\
		0.060  &         10(10) &   1.18(*) &   1.00 &  (10)15.85 &            10(10) &   0.46(*) &   0.30 &   (10)7.96 \\
		0.080  &         10(10) &   1.57(*) &   1.35 &  (10)17.63 &            10(10) &   0.48(*) &   0.42 &   (10)4.31 \\
		0.100  &         10(10) &   1.21(*) &   0.98 &  (10)19.58 &            10(10) &   0.43(*) &   0.38 &   (10)1.99 \\
		0.120  &         10(10) &   1.15(*) &   0.90 &  (10)21.61 &            10(10) &   0.48(*) &   0.48 &   (10)0.10 \\
		0.140  &         10(10) &   1.08(*) &   0.81 &  (10)21.53 &            10(10) &   0.37(*) &   0.37 &   (10)0.00 \\
		0.160  &         10(10) &   0.83(*) &   0.56 &  (10)19.89 &            10(10) &   0.34(*) &   0.34 &   (10)0.00 \\
		0.180  &         10(10) &   0.81(*) &   0.57 &  (10)19.41 &            10(10) &   0.32(*) &   0.32 &   (10)0.08 \\
		\bottomrule
	\end{tabular}
\end{table}

\begin{table}[h!t]
	\centering
	\caption{Results for portfolio for $N=300$}
	\label{tab:port-full-N300}
	\begin{tabular}{l|rrrr|rrrr}
		\toprule
		{} & \multicolumn{4}{c}{\texttt{Basic}} & \multicolumn{4}{c}{\texttt{Improved}} \\
		{$\theta$} &       Slv(Fnd) &     Time(Gap) & R.time & (Fnd)R.gap &          Slv(Fnd) &     Time(Gap) & R.time & (Fnd)R.gap \\
		\midrule
		0.001  &          0(10) &      *(13.51) &   0.36 &  (10)23.09 &             0(10) &       *(8.61) &   0.14 &  (10)17.80 \\
		0.020  &          5(10) &  479.47(9.38) &   0.57 &  (10)23.48 &             7(10) &  556.67(4.22) &   1.00 &  (10)32.42 \\
		0.040  &          7(10) &  685.71(4.64) &   0.50 &  (10)26.73 &            10(10) &     192.80(*) &   1.44 &  (10)20.88 \\
		0.060  &         10(10) &     569.36(*) &   1.58 &  (10)30.69 &            10(10) &      30.27(*) &   1.84 &  (10)16.04 \\
		0.080  &         10(10) &     214.67(*) &   2.37 &  (10)29.59 &            10(10) &      10.54(*) &   1.97 &  (10)11.48 \\
		0.100  &         10(10) &     128.98(*) &   2.71 &  (10)27.37 &            10(10) &       5.78(*) &   1.94 &   (10)7.16 \\
		0.120  &         10(10) &      51.87(*) &   2.41 &  (10)26.34 &            10(10) &       4.53(*) &   1.87 &   (10)5.89 \\
		0.140  &         10(10) &      16.90(*) &   2.48 &  (10)24.88 &            10(10) &       3.07(*) &   1.69 &   (10)4.58 \\
		0.160  &         10(10) &      24.67(*) &   2.52 &  (10)23.75 &            10(10) &       2.79(*) &   1.92 &   (10)4.51 \\
		0.180  &         10(10) &      21.71(*) &   2.67 &  (10)22.22 &            10(10) &       2.29(*) &   1.67 &   (10)3.46 \\
		\bottomrule
	\end{tabular}
\end{table}

\section{Big-$M$ Computation for Resource Planning}\label{app:big-M-resource-planning}

In this appendix we describe our big-$M$ calculation used in~\eqref{eq:rsrcp-re} of the probabilistic resource planning problem studied in~\cref{sec:resource-planning}.
Recall that each sample data $(\vrho^i,\vmu^i,\vlambda^i)$ is nonnegative almost surely. Recall also that the domain $\cX$ of~\eqref{eq:rsrcp-re} is not bounded, so we need to choose a value for $M^i$ based on~\eqref{eq:bigMvalue-joint-opt}, i.e., for some optimal $(\vx,\vy)$ to \cref{eq:rsrcp-re}, $M^i$ must be selected to be greater than or equal to the quantity in~\eqref{rsrcp:M-lb}.
Next, we provide an upper bound on~\eqref{rsrcp:M-lb}, thereby providing a value for $M^i$.

\begin{remark}\label{remark:rsrcp-M}
Let $(\vx,\vy)$ be an optimal solution to~\eqref{eq:rsrcp-re}. Note that if $\mu_{dp}^i=0$ for all $i\in[N]$, we may assume that $y_{dp}=0$, for otherwise, reducing $y_{dp}=0$ does not affect~\eqref{rsrcp:bigM2-p}, and is less restrictive for~\eqref{rsrcp:bigM2-d}. Consider a pair of $d\in[D]$ and $p\in[P]$ such that $y_{dp}>0$. Since $(\vx,\vy)$ satisfies the nominal chance constraint with nonzero probability, there exists $i\in [N]$ such that $\sum_{d\in [D]}\mu_{dp}^iy_{dp}-\lambda_p^i\geq0$. In fact, we may assume that there exists $j\in [N]$ such that $\mu_{dp}^{j}>0$ and equality $\sum_{d\in [D]}\mu_{dp}^{j}y_{dp}-\lambda_p^{j}=0$ holds, for otherwise, one can slightly reduce $y_{dp}$ without affecting the validity of $(\vx,\vy)$. Hence, it follows that $y_{dp}\leq \lambda_p^{j}/\mu_{dp}^{j}$ if $y_{dp}>0$. Let $U_{dp}$ be defined as
\[
U_{dp}=\begin{cases}
\max\left\{{\lambda_p^{i}}/{\mu_{dp}^{i}}:\ \mu_{dp}^{i}>0,\ i\in[N]\right\},&\text{if}\ \mu_{dp}^{i}>0 \ \text{for some} \ i\in[N],\\
0,&\text{otherwise}.
\end{cases}
\]
Then, for every $d\in[D]$ and $p\in[P]$, we have $0\leq y_{dp}\leq U_{dp}$. This implies that
\begin{equation}\label{rsrcp:M-lb1}
\max_{p\in[P]}\bigg\{\bigg|\sum_{d\in [D]}\mu_{dp}^iy_{dp}-\lambda_p^i\bigg|\bigg\}\leq\max\limits_{p\in[P]}\bigg\{\max\bigg\{\lambda_p^i,\ \sum_{d\in [D]}\mu_{dp}^iU_{dp}-\lambda_p^i\bigg\}\bigg\}.
\end{equation}
Now let us consider the other term inside~\eqref{rsrcp:M-lb}. Since $(\vx,\vy)$ satisfies the nominal chance constraint with nonzero probability, there exists $i\in [N]$ such that $\rho_d^ix_d-  \sum_{p\in [P]}y_{dp}\geq0$, and as before, we may assume that equality $\rho_d^jx_d-  \sum_{p\in [P]}y_{dp}=0$ holds for some $j\in[N]$. Hence, it follows that for $i\in[N]$, $
\rho_d^ix_d-  \sum_{p\in [P]}y_{dp}=({\rho_d^i}/{\rho_d^j}-1)\sum_{p\in [P]}y_{dp}.$
Let $\rho_d^{\max}:=\max\{\rho_d^j:\ j\in[N]\}$ and $\rho_d^{\min}:=\min\{\rho_d^j:\ j\in[N]\}$. Then
\begin{equation}\label{rsrcp:M-lb2}
\max\limits_{d\in[D]}\bigg\{\bigg|\rho_d^ix_d-  \sum_{p\in [P]}y_{dp}\bigg|\bigg\}\leq\max\limits_{d\in[D]}\bigg\{\max\bigg\{1-\frac{\rho_d^i}{\rho_d^{\max}},\ \frac{\rho_d^i}{\rho_d^{\min}}-1\bigg\}\cdot\sum_{p\in[P]}U_{dp}\bigg\},
\end{equation}
and $M^i$ can be set to the maximum of the two values given in the right-hand sides of~\eqref{rsrcp:M-lb1} and~\eqref{rsrcp:M-lb2}.
\ifx\flagJournal\true \epr \fi
\end{remark}

While the demand  constraints in~\eqref{eq:resource-panning} are covering type and so we can use~\cref{lemma:cov-pack-quantiles} to improve~\eqref{rsrcp:bigM2-p} by reducing $M^i$, the resource assignment constraints in~\eqref{eq:resource-panning} are neither covering nor packing type, hence we cannot apply~\cref{lemma:cov-pack-quantiles} to compute the reduced coefficient for~\eqref{rsrcp:bigM2-d}. Instead, we compute the reduced coefficient for~\eqref{rsrcp:bigM2-d} based on~\eqref{eq:lower-bounds-joint} as follows. 
\begin{remark}\label{remark:rsrcp-yield-lower-bounds}
By \cref{remark:rsrcp-M}, at optimality we have $0 \leq y_{d'p} \leq U_{d'p}$ for all $d' \in [D]$, $p \in [P]$. Then, for $i,j\in[N]$,
\begin{align}
&\bar h_d^j(-\vA\vxi_d^i- \va_d)\label{eq:rsrcp-quantiles-bound}\\
&\geq\min\limits_{\vx\geq\bm{0}, \vy\geq\bm{0}}\left\{\rho_d^ix_d-  \sum_{p\in [P]}y_{dp}: \begin{array}{l}
\sum_{d'\in[D]}\mu_{d'p}^jy_{d'p}\geq\lambda_p^j, \ p\in[P],\\ \rho_d^jx_d\geq \sum_{p\in [P]}y_{dp},\\
y_{d'p} \leq U_{d'p}, \ d' \in [D], \ p \in [P]
\end{array}
\right\}\notag\\
&\geq\min\limits_{\vx\geq\bm{0}, \vy\geq\bm{0}}\left\{\left(\frac{\rho_d^i}{\rho_d^j}-1\right)\sum_{p\in [P]}y_{dp}: \begin{array}{l}
\sum_{d'\in[D]}\mu_{d'p}^jy_{d'p}\geq\lambda_p^j, \ p\in[P],\\
y_{d'p} \leq U_{d'p}, \ d' \in [D], \ p \in [P]
\end{array}\right\}.\notag
\end{align}
When $({\rho_d^i}/{\rho_d^j}-1)\geq 0$, we set $y_{d'p} = U_{d'p}$ for $d' \neq d$. Then for each $p \in [P]$, we set
\[ y_{dp} = L_{dp}^j := \begin{cases}
\max\left\{0,\lambda_p^j - \sum_{d'\in[D], d' \neq d} U_{d'p}\right\}/{\mu_{dp}^j}, &\mu_{dp}^j > 0\\
0, & \mu_{dp}^j = 0.
\end{cases} \]
It follows from~\eqref{eq:rsrcp-quantiles-bound} that $\bar h_d^j(-\vA\vxi_d^i- \va_d)\geq({\rho_d^i}/{\rho_d^j}-1)\sum_{p\in[P]}L_{dp}^j$ when $({\rho_d^i}/{\rho_d^j}-1)\geq 0$. When $({\rho_d^i}/{\rho_d^j}-1)<0$, since $0 \leq y_{dp}\leq U_{dp}$ at optimality, we obtain $\bar h_d^j(-\vA\vxi_d^i- \va_d)\geq({\rho_d^i}/{\rho_d^j}-1)\sum_{p\in[P]}U_{dp}$. Based on these lower bounds on $\bar h_d^j(-\vA\vxi_d^i- \va_d)$, we can compute a lower bound on $q_p^i$. Note that in the definition of $L_{pd}^j$, if  $\lambda_p^j - \sum_{d'\in[D], d' \neq d} U_{d'p} > 0$ but $\mu_{dp}^j = 0$, then scenario $j$ is infeasible, so we can set $z_j = 1$.
\ifx\flagJournal\true \epr \fi
\end{remark}

\section{Computational Results for Mixing Inequalities}\label{app:mixing}

\cref{tab:mixing} presents computational results for \texttt{Mixing} described in \cref{sec:experiments}.

\begin{table}[h!]
	\centering
	\caption{Results for \texttt{Mixing}.}
	\label{tab:mixing}
	\begin{tabular}{cll|rrrrr}
		\toprule
		& $N$ & $\theta$ &        Slv(Fnd) &      Time(Gap) & R.time & R.gap(Fnd) &    Cuts \\
		\midrule
		\multirow{4}{*}{\rotatebox[origin=c]{90}{Portfolio}} & \multirow{2}{*}{500} & 0.001 &           0(10) &   *(6.10) &   0.62 &  16.99(10) &   89.6 \\
		& & 0.020 &           0(10) &  *(12.74) &   2.54 &  34.15(10) &    0.0 \\
		\cline{2-8}
		& \multirow{2}{*}{1000} & 0.001 &           0(10) &   *(8.79) &   2.28 &  16.43(10) &  271.5 \\
		& & 0.020 &            0(9) &  *(17.52) &   3.86 &   23.46(9) &    0.2 \\
		\cline{1-8}
		\multirow{4}{*}{\rotatebox[origin=c]{90}{Res. plan.}} & \multirow{2}{*}{100} & 0.0001 &           9(10) &   800.84(6.53) &   8.11 &       *(0) &   687.3 \\
		& & 0.0010 &           1(10) &  2307.23(7.99) &  11.66 &       *(0) &    12.5 \\
		\cline{2-8}
		& \multirow{2}{*}{300} & 0.0001 &           0(10) &        *(7.55) &  39.78 &   10.28(5) &  4337.9 \\
		& & 0.0010 &           0(10) &       *(11.88) &  50.85 &       *(0) &   223.5 \\
		\bottomrule
	\end{tabular}
\end{table}

\section{Other Relaxations of \eqref{eq:dr-ccp}}\label{sec:other_relaxations}

\citet[Theorem 3]{xie2018distributionally} provides a better relaxation of~\eqref{eq:dr-ccp} than the ordinary \eqref{eq:ccp}. %
However, we present an argument that the relaxation in \citep[Theorem 3]{xie2018distributionally} cannot be linked with \eqref{eq:dr-ccp} in the same manner as what we do in this paper. To illustrate this, let us consider an individual chance constraint with closed safety set
\[ \cS^c(x) = \left\{ \xi : (b-A^\top x)^\top \xi + d - a^\top x \geq 0 \right\}. \]
The binary variables $\{z^i\}_{i \in [N]}$ appearing in \eqref{eq:joint}, the formulation for \eqref{eq:dr-ccp}, essentially model the disjunction
\[\underbrace{(b-A^\top x)^\top \xi^i + d - a^\top x \geq 0}_{z^i=0} \text{ or } \underbrace{(b-A^\top x)^\top \xi^i + d - a^\top x < 0}_{z^i=1}.\]
The intuitive reason why we can link \eqref{eq:saa-reformulation} and \eqref{eq:joint} is because the binary variables in \eqref{eq:saa-reformulation} model the exact same disjunction. On the other hand, \citep[Theorem 3]{xie2018distributionally} states that a valid inequality for $\inf_{\bbP \in \cF_N(\theta)} \bbP[(b-A^\top x)^\top \xi + d - a^\top x \geq 0] \geq 1-\epsilon$ is the modified nominal chance constraint
\[ \bbP_N\left[ (b-A^\top x)^\top \xi + d - a^\top x \geq \frac{\theta}{\epsilon} \|b-A^\top x\|_* \right] \geq 1-\epsilon. \]
As stated in \citep[Corollary 4]{xie2018distributionally}, this can also be formulated as a MIP, with binary variabes $u^i$ that model the disjunction
\[\underbrace{(b-A^\top x)^\top \xi^i + d - a^\top x \geq \frac{\theta}{\epsilon} \|b-A^\top x\|_*}_{u^i=0} \text{ or } \underbrace{(b-A^\top x)^\top \xi^i + d - a^\top x < \frac{\theta}{\epsilon} \|b-A^\top x\|_*}_{u^i=1}.\]
This is a fundamentally different disjunction than the one in \eqref{eq:joint}. Indeed, if we have a scenario $i$ for which
\[ 0 < (b-A^\top x)^\top \xi^i + d - a^\top x < \frac{\theta}{\epsilon} \|b-A^\top x\|_*, \]
then $z^i = 0$ but $u^i=1$. Therefore, we cannot link the binary variables $z^i$ and $u^i$ in the same manner as we did for the nominal CCP formulation, since there is no a priori reason to prevent such a scenario occurring. Thus,  the strengthening procedure from our paper cannot be adapted to work with the relaxation from \citep{xie2018distributionally}.

\end{document}